\documentclass{amsart}
\usepackage{amssymb}
\usepackage{amsthm}
\usepackage{amsmath,amscd}
\usepackage[mathscr]{euscript}
\usepackage[all]{xy}
\usepackage[utf8]{inputenc}
\usepackage[textwidth=15cm,hcentering]{geometry}
\usepackage{lmodern}
\usepackage[T1]{fontenc}
\usepackage[breaklinks=true]{hyperref}

\title{A model structure on internal categories in simplicial sets}

\author{Geoffroy Horel}

\address{Mathematisches Institut\\
Einsteinstrasse 62\\
D-48149 Münster\\
Deutschland}

\email{geoffroy.horel@gmail.com}

\thanks{The author was supported by Michael Weiss's Humboldt professor grant.}


\newtheorem{theo}{Theorem}[section]
\newtheorem{lemm}[theo]{Lemma}
\newtheorem{prop}[theo]{Proposition}
\newtheorem{coro}[theo]{Corollary}

\theoremstyle{definition}
\newtheorem{defi}[theo]{Definition}

\newtheorem{example}[theo]{Example}
\newtheorem{rem}[theo]{Remark}

\numberwithin{equation}{section}

\newcommand{\op}{^{\mathrm{op}}}
\newcommand{\cat}{\mathbf}
\newcommand{\on}{\operatorname}
\newcommand{\Cat}{\cat{Cat}}
\newcommand{\ICat}{\cat{ICat}}
\newcommand{\id}{\mathrm{id}}
\newcommand{\Int}{\on{Int}}
\newcommand{\Ob}{\on{Ob}}
\newcommand{\Ar}{\on{Ar}}
\newcommand{\Sp}{\cat{S}}
\renewcommand{\L}{\mathbb{L}}
\newcommand{\R}{\mathbb{R}}
\newcommand{\Map}{\on{Map}}
\newcommand{\map}{\on{map}}
\newcommand{\goto}[1]{\stackrel{#1}{\longrightarrow}}

\begin{document}

\maketitle

\begin{abstract}
We put a model structure on the category of categories internal to simplicial sets. The weak equivalences in this model structure are preserved and reflected by the nerve functor to bisimplicial sets with the complete Segal space model structure. This model structure is shown to be a model for the homotopy theory of infinity categories. We also study the homotopy theory of internal presheaves over an internal category.
\end{abstract}

\section*{Introduction}

Infinity-categories are category-like objects in which one can do homotopy theory. There are nowadays a plethora of available definitions of infinity-categories in the literature. The most famous are quasicategories, complete Segal spaces, simplicial categories, Segal categories, relative categories. Each one of these models is organized into a model category which gives a structured way to encode the homotopy theory of infinity-categories. It has been shown by various people (Bergner, Joyal and Tierney, Barwick and Kan, Lurie) that any two of the above models are connected by a zig-zag of Quillen equivalences meaning that all these models are equivalent. The relevant references are~\cite{joyalquasi,lurietopos,bergnerthree,rezkmodel,joyaltierneyquasi,barwickrelative}

The goal of this paper is to introduce yet another model category presenting the homotopy theory of infinity categories. It is a model structure on the category of categories internal to simplicial sets. An internal category in simplicial sets is a diagram of simplicial sets $\Ar(C)\rightrightarrows \Ob(C)$ together with a unit map $\Ob(C)\to \Ar(C)$ and a composition map $\Ar(C)\times_{\Ob(C)}\Ar(C)\to \Ar(C)$ which suitably generalizes the notion of a category. Equivalently, an internal category in simplicial sets is a simplicial object in the category of small categories. Applying the nerve functor degreewise, we can see the category of internal categories as a full subcategory of the category of bisimplical sets. We define a morphism between internal categories to be a weak equivalence if it is sent to one in the model structure of complete Segal spaces. We show that those maps are the weak equivalences of a model structure. This model structure is transferred from the \emph{projective} model structure of complete Segal spaces (as opposed to the injective model structure used in~\cite{rezkmodel}). This result answers a question of Mike Shulman on Mathoverflow (see~\cite{shulmaninternal}). This model structure inherits some of the good formal properties of the model category of complete Segal spaces. In particular, it is a left proper simplicially enriched model category. 

In this paper, we also study the homotopy theory of internal presheaves over a fixed internal category. We put a model structure on this category which generalizes the projective model structure on simplicial presheaves over a simplicial category. We also prove that this model structure is homotopy invariant in the sense that a weak equivalence of internal categories induces a Quillen equivalence of the presheaf categories.

There are many interesting examples of internal categories. For instance Rezk in~\cite{rezkmodel} defines a nerve from relative categories to bisimplicial sets and this functor factors through the category of internal categories. In particular, the main result of~\cite{barwickpartial} shows that a levelwise fibrant replacement of the Rezk nerve of a partial model category is a fibrant internal category in our sense. Simplicially enriched categories are also particular internal categories and we show that the inclusion of the category of simplicial categories in the category of internal categories preserves the class of weak equivalences on both categories and induces an equivalence of the underlying infinity-categories. 

Another source of examples comes from the Grothendieck construction of simplicial presheaves. If $C$ is a simplicially enriched category and $F$ is a presheaf on $C$ with value in simplicial sets, the Grothendieck construction of $F$ is very naturally an internal category. Indeed, we can declare $Gr(F)$ to be the internal category
\[Gr(F)=\bigsqcup_{c,d \in\Ob(C)}F(c)\times \map_C(c,d)\rightrightarrows \bigsqcup_{c\in\Ob(C)}F(c)\]
where the source map is given by the projection and the target map is given by the action of $C$ on $F$. To our knowledge, there is no good model for the Grothendieck construction which remains in the world of simplicial categories. 

\subsection*{Overview of the paper.}

The first section contains a few reminders on model categories and their left Bousfield localizations.

The second section describes a projective version of Rezk's model structure of complete Segal spaces. It is a model category structure on simplicial spaces which is Quillen equivalent to Rezk's model category of complete Segal spaces but in which the cofibrations are the projective cofibrations (i.e. the maps with the left lifting property against levelwise trivial fibrations) as opposed to the injective cofibrations that are used in~\cite{rezkmodel}. We study the fibrant objects in this model structure (proposition~\ref{prop-Segal fibrant} and proposition~\ref{prop-Rezk fibrant}) and we generalize the theory of Dwyer-Kan equivalences in this context (proposition~\ref{prop-Dwyer Kan equivalences are Rezk equivalences}).

The third section contains background material on the main object of the paper, namely the category $\ICat$ of internal categories in the category of simplicial sets.

The fourth section is a proof of a technical lemma (lemma~\ref{lemm-key lemma}) that is the key step in the proof of the existence of the model structure on $\ICat$.

The fifth section contains the construction of the model structure on $\ICat$ and the proof of the equivalence with the model category of complete Segal spaces. The main theorem is theorem~\ref{theo-main}.

The sixth section studies the category of internal presheaves on an internal category. In good cases, we put a model structure on this category which generalizes the projective model structure on simplicial presheaves on a simplicial category. We also show that a map between internal categories induces a Quillen adjunction between the categories of internal presheaves, and that this adjunction is a Quillen equivalence if the map of internal categories was a weak equivalence (see theorem~\ref{theo-invariance of presheaf category}).

The seventh section is devoted to the study of the inclusion functor from simplicially enriched categories to internal categories. This functor is not a Quillen functor but we prove (theorem~\ref{theo-Int is an equivalence}) that it induces an equivalence between the infinity-category of simplicial categories and the infinity-category of internal categories. 

\subsection*{Acknowledgements.}

I wish to thank Pedro Boavida de Brito, Matan Prezma and Clark Barwick for helpful conversations. I also want to thank the anonymous referee, Dimitri Ara, Viktoriya Ozornova and Mike Shulman for several useful comments on the first drafts of this paper.

\subsection*{Notations.}

We write $\Sp$ for the category of simplicial sets. We often say space instead of simplicial set. The category $\Sp$ will always be equipped with its standard model structure. The points of an object $X$ of $\Sp$ are by definition the $0$-simplices of $X$.

We write $s\Sp$ for the category of simplicial objects in $\Sp$. We implicitly identify the category $\Sp$ with the full subcategory of $s\Sp$ on constant diagrams.

The category $\Cat$ is the category of small categories.

If $\cat{C}$ is a category and $c$ is an object of $\cat{C}$, we denote by $\cat{C}_{/c}$ the overcategory of $c$.

For $k$ a natural number, we denote by $[k]$ the poset $\{0 \leq 1\leq\ldots\leq k\}$ seen as an object of $\Cat$. The object $\Delta[k]$ in $\Sp$ is the object representing the functor $X\mapsto X_k$. The object $\Delta[k]$ is the nerve of the discrete category $[k]$. We usually write $*$ instead of $\Delta[0]$.

We denote by $F(k)$ the functor $\Delta\op\to \Sp$ sending $[n]$ to the discrete simplicial set $\Cat([n],[k])$.

We generically denote by $\cong$ an isomorphism and by $\simeq$ a weak equivalence in the ambient model category. 

We generically denote by $Q$ and $R$ a cofibrant and fibrant replacement functor. In this paper all model categories are cofibrantly generated which ensures that $Q$ and $R$ exist.

If $F$ is a left Quillen functor, we denote by $\mathbb{L} F$ the functor $F\circ Q(-)$ where $Q$ is any cofibrant replacement functor in the source of $F$. By Ken Brown's lemma this is well-defined up to a weak equivalence. Similarly, if $G$ is a right Quillen functor, we denote by $\mathbb{R} G$ the functor $G\circ R$ where $R$ is any fibrant replacement functor in the source of $G$. Note that if $G$ happens to preserve all weak equivalences, then $\R G$ is weakly equivalent to $G$. We implicitly use this fact in various places in this paper.

\subsection*{Nine model categories.}

To help the reader keep track of the various model categories defined in this paper, we have the following diagram of right Quillen functors. In this diagram, all the horizontal functors are right Quillen equivalences which preserve and reflect weak equivalences and the vertical functors are right adjoint to left Bousfield localizations. The number next to each category refers to the section where the model structure is defined.

\[
\xymatrix{
s\Sp_{inj}\ar[r]&s\Sp_{proj}(\S\ref{ss projective})&\ICat^{LW}(\S\ref{Icat projective}\ar[l])\\
\cat{SS}_{inj}\ar[r]\ar[u]&\cat{SS}_{proj}(\S\ref{ss Segal})\ar[u]&\ICat^S(\S\ref{Icat Segal})\ar[l]\ar[u]\\
\cat{CSS}_{inj}\ar[r]\ar[u]&\cat{CSS}_{proj}(\S\ref{ss Rezk})\ar[u]&\ICat(\S\ref{Icat Rezk})\ar[l]\ar[u]
}
\]

\section{A few facts about model categories}

\subsection{Cofibrant generation.}

The following definition is standard terminology.

\begin{defi}
Let $\cat{X}$ be a cocomplete category and $I$ a set of maps in $\cat{X}$. The $I$-cell complexes are the elements of the smallest class of maps in $\cat{X}$ containing $I$ and closed under pushout and transfinite composition. The $I$-fibrations are the maps with the right lifting property against $I$. The $I$-cofibrations are the maps with the left lifting property against the $I$-fibrations. 
\end{defi}

Recall that the $I$-cofibrations are the retracts of $I$-cell complexes. One also shows that the $I$-fibrations are exactly the maps with the right lifting property against the $I$-cofibrations. All these facts can be found in appendix A of~\cite{lurietopos}.

A model category $\cat{C}$ is said to be cofibrantly generated if there are sets $I$ and $J$ in $\cat{C}^{[1]}$ whose members have a small source and such that the fibrations of $\cat{C}$ are the $J$-fibrations and the trivial fibrations are the $I$-fibrations. Recall that a cofibrantly generated model category has functorial factorizations given by the small object argument. In particular, it has a cofibrant replacement functor and a fibrant replacement functor.

A model category $\cat{C}$ is said to be combinatorial if its underlying category is locally presentable and if it is cofibrantly generated.

For future reference, we recall the following classical theorem of transfer of model structures:

\begin{theo}\label{theo-transferred model structure}
Let $F:\cat{X}\leftrightarrows\cat{Y}:U$ be an adjunction between complete and cocomplete categories where $\cat{X}$ has a cofibrantly generated model structure in which the set of generating cofibrations (resp. trivial cofibrations) is denoted by $I$ (resp. $J$). Assume that 
\begin{itemize}
\item $U$ preserves filtered colimits.
\item $U$ sends pushouts of maps in $FI$ to $I$-cofibrations and pushouts of maps in $FJ$ to $J$-cofibrations.
\end{itemize}
Then there is a model structure on $\cat{Y}$ whose fibrations (resp. weak equivalences) are the maps that are sent to fibrations (resp. weak equivalences) by $U$. Moreover, the functor $U$ preserves cofibrations.
\end{theo}

\begin{proof}
This is proved for instance in~\cite[Proposition 11.1.4]{fressemodules}.
\end{proof}

\subsection{Simplicial model categories.}

All the model categories in this work will be simplicial model categories. If $\cat{C}$ is a simplicial model category, we denote by $\Map_{\cat{C}}(-,-)$, or $\Map(-,-)$ if there is no possible ambiguity, the bifunctor $\cat{C}\op\times\cat{C}\to\Sp$ giving the simplicial enrichment.

If $\cat{C}$ is a simplicial model category, Ken Brown's lemma implies that the functor $\Map_{\cat{C}}(-,-)$ preserves weak equivalences between pairs of objects of $\cat{C}$ whose first component is cofibrant and second component is fibrant. If $\cat{C}$ is cofibrantly generated, we write $\mathbb{R}\Map_{\cat{C}}(-,-)$ for the functor $\Map_{\cat{C}}(Q-,R-)$ where $Q$ and $R$ denote respectively a cofibrant and fibrant replacement functor in $\cat{C}$. 

If $\cat{C}$ is a simplicial cofibrantly generated model category, then it admits a simplicial cofibrant replacement functor and a simplicial fibrant replacement functor (see for instance~\cite[Theorem 6.1.]{blumberghomotopical}). In the following, we will always assume that $Q$ and $R$ are simplicial which implies that the functor $\mathbb{R}\Map_{\cat{C}}(-,-)$ is a simplicial and weak equivalence preserving functor $\cat{C}\op\times\cat{C}\to\Sp$.

\begin{prop}\label{prop-derived adjunction}
Let $F:\cat{C}\leftrightarrows\cat{D}:G$ be a simplicial Quillen adjunction between cofibrantly generated simplicial model categories. Then, in the category of $\Sp$-enriched functors from $\cat{C}\op\times\cat{D}$ to $\Sp$, there is a zig-zag of natural transformations between $\mathbb{R}\Map_{\cat{D}}(\mathbb{L}F-,-)$ and $\mathbb{R}\Map_{\cat{C}}(-,\mathbb{R}G-)$  which is objectwise a weak equivalence.
\end{prop}

\begin{proof} Recall that $\mathbb{L}F=FQ$ and $\mathbb{R}G=GR$. The natural zig-zag is given by
\[\mathbb{R}\Map_{\cat{D}}(FQ-,-)=\Map_{\cat{D}}(QFQ-,R-)\leftarrow \Map_{\cat{D}}(FQ-,R-)\cong\Map_{\cat{C}}(Q-,GR-)\]
\[\to\Map_{\cat{C}}(Q-,RGR-)=\mathbb{R}\Map_{\cat{C}}(-,GR-)\]
in which the backward arrow is given by the natural transformation $Q\to\id_{\cat{D}}$, the forward arrow is given by the natural transformation $\id_{\cat{C}}\to R$ and the middle isomorphisms comes from the fact that $(F,G)$ is a simplicial adjunction. Using the fact that $F$ is a left Quillen functor, we find that $FQX$ is cofibrant for any $X$ which forces the backward map to be objectwise a weak equivalence. Similarly, using the fact that $G$ is right Quillen, we show that the forward map is objectwise a weak equivalence.
\end{proof}

Now, we want to prove that the property of being simplicial for a model category is preserved under transfer along simplicial adjunction.

\begin{prop}\label{prop-transferred is simplicial}
Let $F:\cat{X}\leftrightarrows\cat{Y}:U$ be a simplicial adjunction that satisfies the hypothesis of theorem~\ref{theo-transferred model structure}. Then the model structure on $\cat{Y}$ is simplicial.
\end{prop}

\begin{proof}
Let $I$ (resp. $J$) be a set of generating cofibrations (resp. trivial cofibrations) of $\cat{X}$. The model structure on $\cat{Y}$ has $FI$ (resp. $FJ$) as generating cofibrations (resp. trivial cofibrations). Indeed, it is obvious that the fibrations (resp. trivial fibrations) are the maps with the right lifting property against $FJ$ (resp. $FI$) moreover, since $U$ preserves filtered colimits, the sources of the maps in $FI$ and $FJ$ are small. Now, we prove that $\cat{Y}$ is simplicial. Since the mapping spaces $\Map_{\cat{Y}}(-,-)$ preserve colimits in the first variable, it suffices to check that for each generating cofibration $f:C\to D$ and fibration $E\to F$ in $\cat{Y}$, the map
\[\Map_{\cat{Y}}(D,F)\to\Map_{\cat{Y}}(C,F)\times_{\Map_{\cat{Y}}(C,E)}\Map_{\cat{Y}}(D,E)\]
is a fibration. But $f:C\to D$ is $F(g)$ for $g:A\to B$ an element of $I$. Therefore, we want to prove that
\[\Map_{\cat{Y}}(FB,F)\to\Map_{\cat{Y}}(FA,F)\times_{\Map_{\cat{Y}}(FA,E)}\Map_{\cat{Y}}(FB,E)\]
is a fibration. Using the fact that the adjunction $(F,U)$ is simplicial, this map is isomorphic to 
\[\Map_{\cat{X}}(B,UF)\to\Map_{\cat{X}}(A,UF)\times_{\Map_{\cat{X}}(A,UE)}\Map_{\cat{X}}(B,UE)\]
which is a fibration by our assumption that $\cat{X}$ is a simplicial model category. The case where the map $C\to D$ is a trivial cofibration or the map $E\to F$ is a trivial fibration is treated analogously.
\end{proof}

\subsection{Bousfield localization.}

\begin{defi}
Let $\cat{X}$ be a simplicial model category and $S$ a set of arrows in $\cat{X}$. We say that an object $Z$ of $\cat{X}$ is $S$-local if for all $u:A\to B$ in $S$, the induced map
\[\R\Map_{\cat{X}}(B,Z)\to \R\Map_{\cat{X}}(A,Z)\]
is a weak equivalence.
\end{defi}

For future reference, we recall the following theorem:

\begin{theo}\label{theo-existence Bousfield localization}
Let $\cat{X}$ be a combinatorial left proper simplicial model category and let $S$ be a set of arrows in $\cat{X}$. There is a model structure on $\cat{X}$ denoted $L_S\cat{X}$ satisfying the following properties.
\begin{itemize}
\item The cofibrations of $L_S\cat{X}$ are the cofibrations of $\cat{X}$.
\item The fibrant objects of $L_S\cat{X}$ are the fibrant objects of $\cat{X}$ that are also $S$-local.
\item The weak equivalences of $L_S\cat{X}$ are the maps $f:X\to Y$ such that the induced map
\[\R\Map_X(Y,K)\to\R\Map(X,K)\]
is a weak equivalence in $\Sp$ for every $S$-local object $K$.
\end{itemize}
Moreover, $L_S\cat{X}$ is left proper, combinatorial, and if $\cat{X}$ admits a set of generating cofibrations with cofibrant source, then $L_S\cat{X}$ is simplicial.
\end{theo}

\begin{proof}
This is proved in~\cite[Theorem 4.7. and Theorem 4.46]{barwickleft}.
\end{proof}

For future reference, we have the following proposition which explains how Bousfield localization interacts with certain Quillen equivalences.

\begin{prop}\label{prop-Bousfield localization}
Let $F:\cat{X}\leftrightarrows\cat{Y}:G$ be a Quillen equivalence. Let $S$ be a set of maps in $\cat{X}$ and let $L_S\cat{X}$ (resp. $L_{\mathbb{L}FS}\cat{Y}$) be the left Bousfield localization of $\cat{X}$ (resp. $\cat{Y}$) with respect to $S$ (resp. $\mathbb{L}FS$). Then we have a Quillen equivalence
\[F:L_S\cat{X}\leftrightarrows L_{\mathbb{L}FS}\cat{Y}:G\]
Moreover, if the functor $G$ preserves and reflects weak equivalences before localization, it is still the case after localization.
\end{prop}

\begin{proof}
This proposition without the last claim is~\cite[Theorem 3.3.20]{hirschhornmodel}.

For $u:A\to B$ any map in $\cat{X}$ and $Z$ an object of $\cat{X}$, we denote by $u^*$ the map 
\[\R\Map_{\cat{X}}(B,Z)\to\R\Map_{\cat{X}}(A,Z)\]
obtained from the contravariant functoriality of $\R\Map_{\cat{X}}(-,-)$ in the first variable.

Let us assume that $G$ preserves and reflects weak equivalences. We first observe that $\R G$ coincides with $G$ up to weak equivalence. Let $f:U\to V$ be a map in $\cat{Y}$. The map $f$ is a weak equivalence in $L_{\mathbb{L}FS}\cat{Y}$ if and only if for any $\mathbb{L}FS$-local object $Z$ of $\cat{Y}$, the induced map
\[\R\Map_{\cat{Y}}(V,Z)\goto{f^*}\R\Map_{\cat{Y}}(U,Z)\]
is a weak equivalence.

Since $G$ is a Quillen weak equivalence, the counit map $\L FGY\to Y$ is a weak equivalence in $\cat{Y}$ for all $Y$. Therefore, $f$ is a weak equivalence in $L_{\mathbb{L}FS}\cat{Y}$ if and only if for any $\mathbb{L}FS$-local object $Z$ of $\cat{Y}$, the map
\[\R\Map_{\cat{Y}}(\L FGV,Z)\goto{\L FG(f)^*}\R\Map_{\cat{Y}}(\L FGU,Z)\]
is a weak equivalence. Using proposition~\ref{prop-derived adjunction}, this happens if and only if
\[\R\Map_{\cat{X}}(GV,GZ)\goto{G(f)^*}\R\Map_{\cat{X}}(GU,GZ)\]
is a weak equivalence for any $\mathbb{L}FS$-local object $Z$ of $\cat{Y}$.

Thus, in order to prove the proposition, it suffices to prove that the class of $S$-local objects is exactly the class of objects of $\cat{X}$ that are weakly equivalent to one of the form $GZ$ for $Z$ a $\mathbb{L}FS$-local object.

On the one hand, if $Z$ is $\mathbb{L}FS$-local, an application of proposition~\ref{prop-derived adjunction} immediately shows that $GZ$ is $S$-local. 

Let $s:A\to B$ be any map in $S$ and $Z$ be any object of $X$. Then according to proposition~\ref{prop-derived adjunction}, the map
\[\R\Map_{\cat{Y}}(\L FB, \L FZ)\goto{\L F(s)^*} \R\Map_{\cat{X}}(\L FA,\L FZ)\]
is a weak equivalence if and only if the map
\[\R\Map_{\cat{X}}(B,G\L FZ)\goto{s^*} \R\Map_{\cat{X}}(A,G\L FZ)\]
is one. Therefore, $\L F(Z)$ is $\L FS$-local if and only if $G\L FZ$ is $S$-local. But, since $(F,G)$ is a Quillen adjunction, the unit map $Z\to G\L FZ$ is a weak equivalence in $\cat{X}$. This means that the functor $\L F$ sends $S$-local objects to $\L FS$-local objects. In particular, any $S$-local object $X$ is weakly equivalent to $G\L F(X)$ which is of the form $GZ$ for $Z$ an $\L FS$-local object.
\end{proof}

\subsection{Homotopy cartesian squares.}

For future reference, we recall the definition of a homotopy cartesian squares.

Let $W$ be the small category freely generated by the directed graph $0\rightarrow 01 \leftarrow 1$ and let $\Sp^W$ be the functor category. We can give it the injective model structure in which a morphism is a weak equivalence or cofibration if it is levelwise a weak equivalence or cofibration. We denote by $R$ a fibrant replacement functor in $\Sp^W$. For $X=X_0\to X_{01}\leftarrow X_1$ an object of $\Sp^W$, we denote by $X_0\times^h_{X_{01}}X_1$ the pullback of $RX$. We call it the homotopy pullback. Note that there is a map $X_0\times_{X_{01}} X_1\to X_0\times^h_{X_{01}} X_1$
which depends functorially on $X$. The functor $\Sp^W\to \Sp$ sending a span to its homotopy pullback is weak equivalence preserving.

\begin{defi}
A commutative square
\[
\xymatrix{
X_\varnothing\ar[d]\ar[r]&X_0\ar[d]\\
X_1\ar[r]&X_{01}}
\]
is said to be homotopy cartesian if the composite $X_\varnothing\to X_0\times_{X_{01}}X_1\to X_0\times^h_{X_{01}}X_1$ is a weak equivalence.
\end{defi}

\begin{rem}
It is a standard fact about model categories that this definition is independent of the choice of $R$. In fact by right properness of $\Sp$, a commutative square
\[
\xymatrix{
X_\varnothing\ar[d]\ar[r]&X_0\ar[d]^p\\
X_1\ar[r]&X_{01}}
\]
is homotopy cartesian if and only if there exists a factorization of $p$ as a weak equivalence $X_0\to X'_0$ followed by a fibration $X'_0\to X_{01}$ such that the induced map $X_\varnothing\to X_1\times_{X_{01}}X'_0$ is a weak equivalence. 
\end{rem}

If $f:X\to Y$ is a map in $\Sp$ and $y:*\to Y$ is a point in $Y$, we denote by $\on{hofiber}_yf$ the homotopy pullback $X\times^h_Y *$.

We will need the following two classical facts about homotopy cartesian squares.

\begin{prop}\label{prop-characterization homotopy cartesian}
Let 
\[
\xymatrix{
K\ar[r]\ar[d]_p& L\ar[d]^q\\
M\ar[r]_f& N
}
\]
be a square in $\Sp$ in which each corner is fibrant. Then, it is homotopy cartesian if and only if for each point $m$ in $M$, the induced map $\on{hofiber}_mp\to\on{hofiber}_{f(m)}q$ is a weak equivalence.
\end{prop}

\begin{proof}
This is proved in~\cite[Proposition 3.3.18]{munsoncubical}.
\end{proof}

\begin{prop}\label{prop-invariance of homotopy cartesian squares}
Let 
\[\xymatrix{
X_\varnothing\ar[d]\ar[r]&X_0\ar[d]& &Y_\varnothing\ar[r]\ar[d]&Y_0\ar[d]\\
X_1\ar[r]&X_{01}& &Y_1\ar[r]&Y_{01}
}
\]
be two commutative squares in $\Sp$ and let $f$ be a weak equivalence between them in the category of squares of simplicial sets. Then, one of them is homotopy cartesian if and only if the other is homotopy cartesian.
\end{prop}

\begin{proof}
We have a commutative diagram
\[\xymatrix{
X_\varnothing\ar[r]\ar[d]& X_0\times_{X_{01}}X_1\ar[r]\ar[d]& X_0\times^h_{X_{01}}X_1\ar[d]\\
Y_\varnothing\ar[r]& Y_0\times_{Y_{01}}Y_1\ar[r]& Y_0\times^h_{Y_{01}}Y_1
}
\]
in which the vertical maps are induced by $f$. The leftmost and rightmost vertical maps are weak equivalences. Thus, the composite of the two top horizontal maps is a weak equivalence if and only if the composite of the bottom two horizontal maps is a weak equivalence.
\end{proof}

\section{Six model structures on simplicial spaces}

\subsection{The projective model structure.}\label{ss projective}

The category $s\Sp$ can be given the projective model structure. In this model structure, the weak equivalences and fibrations are the maps which are weak equivalences and fibrations in each degree. We denote by $s\Sp_{proj}$ this model category.

A set of generating cofibrations (resp. trivial cofibrations) is given by the maps
\[F(n)\times K\to F(n)\times L\]
where $n$ can be any nonnegative integer and $K\to L$ is any element of a set of generating cofibrations (resp. trivial cofibrations) of $\Sp$.

This model structure is simplicial. For $X$ and $Y$ two objects of $s\Sp$, the space of maps between them is given by:
\[\Map_{s\Sp}(X,Y)_k=s\Sp(X\times\Delta[k],Y)\]
where $\Delta[k]$ denotes the constant simplicial space which is $\Delta[k]$ in each degree.

The model category $s\Sp$ is also proper and combinatorial. Its weak equivalences are stable under filtered colimits.

Let us denote by $s\Sp_{inj}$ the category of simplicial spaces equipped with the injective model structure. This is the model structure in which the cofibrations (resp. weak equivalences) are the maps which are levelwise cofibrations (resp. weak equivalences). The identity map $s\Sp_{proj}\to s\Sp_{inj}$ is a left Quillen equivalence. According to~\cite[Theorem 15.8.7.]{hirschhornmodel}, the injective model structure coincides with the Reedy model structure.

\subsection{The Segal model structure.}\label{ss Segal}

To a simplicial space $X$, we can assign the $n$-fold fiber product $X_1\times_{X_0}\ldots\times_{X_0}X_1$. This defines a simplicial functor from $s\Sp$ to $\Sp$ which is representable by a simplicial space $G(n)$ (see~\cite[section 4.1.]{rezkmodel} for an explicit construction of $G(n)$). There is a map $G(n)\to F(n)$ representing the Segal map
\[X_n\to X_1\times_{X_0}\ldots\times_{X_0}X_1\]

\begin{defi}
The category $\cat{SS}_{proj}$ is the left Bousfield localization of $s\Sp_{proj}$ with respect to the maps $G(n)\to F(n)$ for any $n\geq 1$.
\end{defi}

The existence of this model structure follows from theorem~\ref{theo-existence Bousfield localization} since $s\Sp_{proj}$ is left proper and combinatorial. Moreover, this model structure is simplicial, left proper, combinatorial.

If we denote by $\cat{SS}_{inj}$ the same localization on $s\Sp_{inj}$, we get, by proposition~\ref{prop-Bousfield localization}, a Quillen equivalence
\[\id:\cat{SS}_{proj}\leftrightarrows \cat{SS}_{inj}:\id\]
in which both sides have the same weak equivalences by proposition~\ref{prop-Bousfield localization}.

\subsection{The Rezk model structure.}\label{ss Rezk}

Let $I[1]$ be the category with two objects and one isomorphism between them. Let $E$ be its nerve seen as a levelwise discrete simplicial space.

\begin{defi}
The model category $\cat{CSS}_{proj}$ is the left Bousfield localization of $\cat{SS}_{proj}$ with respect to the unique map $E\to F(0)$.
\end{defi}

This Bousfield localization exists since $\cat{SS}_{proj}$ is left proper and combinatorial. Moreover, this model structure is simplicial, left proper and combinatorial. 

If we denote by $\cat{CSS}_{inj}$ the same localization on $\cat{SS}_{inj}$, we get, by proposition~\ref{prop-Bousfield localization}, a Quillen equivalence
\[\id:\cat{CSS}_{proj}\leftrightarrows \cat{CSS}_{inj}:\id\]
in which both sides have the same weak equivalences by proposition~\ref{prop-Bousfield localization}.

\subsection{The fibrant objects of $\cat{SS}_{proj}$.}

In this subsection, we give an explicit description of the fibrant objects in $\cat{SS}_{proj}$. Recall that the fibrant objects of $\cat{SS}_{inj}$ are called the Segal spaces. By theorem~\ref{theo-existence Bousfield localization}, they are the injectively fibrant simplicial spaces $X$ such that the Segal maps
\[X_n\to X_1\times_{X_0}\times\ldots\times_{X_0}X_1\]
are weak equivalences.

\begin{prop}\label{prop-Segal fibrant}
Let $X$ be a fibrant object of $s\Sp_{proj}$. The following conditions are equivalent.
\begin{enumerate}
\item $X$ is fibrant in $\cat{SS}_{proj}$.
\item $X$ is local with respect to the maps $G(n)\to F(n)$ for all $n$.
\item For each $m$ and $n$, the following commutative diagram is homotopy cartesian.
\[\xymatrix{
X_{m+n}\ar[r]^{l_{m,n}^*}\ar[d]_{r_{n,m}^*}&X_m\ar[d]^{r_m^*}\\
X_n\ar[r]_{l_n^*}& X_0
}
\]
In this diagram, the map $l_{m,n}:[m]\to[m+n]$ sends $i$ to $i$, the map $r_{n,m}:[n]\to[m+n]$ sends $j$ to $j+m$, the map $l_n$ sends the unique object of $[0]$ to $0$ and the map $r_m$ sends the unique object of $[0]$ to $m$.

\item For each $m$, the commutative square
\[\xymatrix{
X_{m+1}\ar[r]\ar[d]&X_m\ar[d]\\
X_1\ar[r]& X_0
}
\]
which is a particular case of the previous one with $n=1$ is homotopy cartesian.
\item For any levelwise weak equivalence $X\to Y$ with $Y$ fibrant in $s\Sp_{inj}$, $Y$ is a Segal space.
\item There exists a levelwise weak equivalence $X\to Y$ with $Y$ a Segal space.
\end{enumerate}
\end{prop}

\begin{proof}
$(1)\Longleftrightarrow(2)$ follows from the characterization of the fibrant objects of left Bousfield localization given in theorem ~\ref{theo-existence Bousfield localization}.

$(2)\implies (3)$ Note that conditions $(2)$ and $(3)$ are invariant under levelwise weak equivalences of simplicial spaces by proposition~\ref{prop-invariance of homotopy cartesian squares}. Thus we can assume that $X$ is fibrant in $s\Sp_{inj}$. 

The map $r_m:X_m=\Map(F(m),X)\to X_0=\Map(F(0),X)$ is represented by a map  $r^m:F(0)\to F(m)$. Similarly $l_n:X_n\to X_0$ is represented by $l^n:F(0)\to F(n)$. It is easy to verify that these maps factor through $G(m)$ and $G(n)$ so that we have a commutative diagram
\[
\xymatrix{
G(m)\ar[d]&G(0)=F(0)\ar[l]\ar[r]\ar[d]_{id}&G(n)\ar[d]\\
F(m)      &F(0)\ar[l]^{r^m}\ar[r]_{l^n}           &F(n)
}
\]
The pushout of the top row is $G(m+n)$ and we denote by $F(m,n)$ the pushout of the bottom row. Since all the vertical maps are weak equivalences in $\cat{SS}_{inj}$, and the horizontal maps are cofibrations in $\cat{SS}_{inj}$, the map $G(m+n)\to F(m,n)$ is a weak equivalence in $\cat{SS}_{inj}$. Note that there is an obvious map $F(m,n)\to F(m+n)$. The composite of that map with the previous map $G(m+n)\to F(m,n)$ is the map $G(m+n)\to F(m+n)$ which is the map representing the Segal map and is thus by construction a weak equivalence in $\cat{SS}_{inj}$. Therefore, by the two-out-of-three property, the map $F(m,n)\to F(m+n)$ is a weak equivalence in $s\Sp_{inj}$. Applying $\Map(-,X)$ to this map, we find that the map $X_{m+n}\to X_m\times_{X_0}X_n$ is a weak equivalence. Since $X$ is injectively fibrant, the maps $X_m\to X_0$ and $X_n\to X_0$ are fibrations which implies that the square
\[\xymatrix{
X_{m+n}\ar[r]\ar[d]&X_m\ar[d]^{r_m^*}\\
X_n\ar[r]_{l_n^*}& X_0
}
\]
is homotopy cartesian.

$(3)\implies (4)$ is immediate.

$(4)\implies (2)$ Again, we can assume that $X$ is fibrant in $s\Sp_{inj}$. We want to prove that it is local with respect to the maps $G(n)\to F(n)$ for each $n$. The case $n=0$ and $n=1$ are trivial. We proceed by induction. We assume that $X$ is local with respect to $G(k)\to F(k)$ for $k\leq m$ and that $X$ satisfies condition $(4)$. According to the proof of $(2)\implies (3)$, condition $(4)$ implies locality of $X$ with respect to the map $F(m,1)\to F(m+1)$. The map $G(m+1)\to F(m+1)$ factors through $F(m,1)\to F(m+1)$. Thus it suffices to check that $X$ is local with respect to $G(m+1)\to F(m,1)$. As in the proof of $(2)\implies (3)$, we have a commutative diagram in $s\Sp$
\[\xymatrix{
G(m)\ar[d]&G(0)=F(0)\ar[l]\ar[r]\ar[d]_{id}&G(1)\ar[d]\\
F(m)      &F(0)\ar[l]^{r^m}\ar[r]_{l^1}           &F(1)
}
\]
such that if we take the pushout of each row, we find the map $G(m+1)\to F(m,1)$. We can apply $\Map(-,X)$ to this diagram and get a diagram
\[\xymatrix{
\Map(G(m),X)\ar[r]&X_0&X_1\ar[l]\\
X_m\ar[u]\ar[r]&X_0\ar[u]&X_1\ar[u]\ar[l]
}
\]
Since $X$ is injectively fibrant, each of the horizontal map is a fibration and by the induction hypothesis, the vertical maps are weak equivalences. Therefore, the induced map on pullbacks is a weak equivalence which is precisely saying that $X$ is local with respect to $G(m+1)\to F(m,1)$.

$(5)\implies (6)$ If $X$ satisfies $(5)$, then we can take $X\to Y$ to be a fibrant replacement in $s\Sp_{inj}$ and $Y$ is a Segal space.

$(6)\implies (5)$ Let $X\to Y$ be a levelwise weak equivalence with $Y$ a Segal space which exists because $X$ satisfies $(6)$. Let $X\to Z$ be a levelwise weak equivalence with $Z$ fibrant in $s\Sp_{inj}$. Then we have a zig-zag of levelwise weak equivalences between $Y$ and $Z$. Since $Y$ is local with respect to $G(n)\to F(n)$ for all $n$, so is $Z$. In particular, by~\ref{theo-existence Bousfield localization}, $Z$ is fibrant in $\cat{SS}_{inj}$ i.e. is a Segal space.

$(6)\Longleftrightarrow (1)$ This follows from~\cite[Proposition 3.6]{colemixing}. Indeed, $\cat{SS}_{proj}$ is the mixed model structure obtained by taking the cofibrations of $s\Sp_{proj}$ and the weak equivalences of $\cat{SS}_{inj}$.
\end{proof}

From now on, a fibrant object of $\cat{SS}_{proj}$ will be called a Segal fibrant simplicial space.

\subsection{The fibrant objects of $\cat{CSS}_{proj}$.}

Now, we give an explicit description of the fibrant objects of $\cat{CSS}_{proj}$. First we construct the space of homotopy equivalences of a Segal fibrant simplicial space

The map of categories $[1]\to I[1]$ induces a map $F(1)\to E$ in $s\Sp$ after taking the nerve. Let $F(1)\to J\to E$ be a factorization of this map as a cofibration followed by a trivial fibration in $s\Sp_{proj}$. Note that since $F(1)$ is cofibrant in $s\Sp_{proj}$, then $J$ is cofibrant as well.

For $K$ a Kan complex (i.e. a fibrant object of $\Sp$), we denote by $\pi_0(K)$ the set of $0$-simplices of $K$ quotiented by the equivalence relation which identifies two $0$-simplices $x$ and $y$ if there is a $1$-simplex $z$ such that $d_0(z)=x$ and $d_1(z)=y$. It is well-known that this coincides with the set of path components of the geometric realization of $K$. Thus, we will call the elements of $\pi_0(K)$ the path components of $K$. Any Kan complex splits as a disjoint union
\[K=\bigsqcup_{x\in\pi_0(K)} K_x\]
where $K_x$ is the simplicial set whose $n$-simplices are the $n$-simplices of $K$ whose vertices are all in $x$. Note that all the spaces $K_x$ are Kan complexes which implies immediately that the obvious map $K\to \pi_0(K)$ is a Kan fibration.

\begin{defi}\label{defi-X_{hoequiv}}
For $X$ a Segal fibrant simplicial space, the space $X_{hoequiv}$ is defined by the following pullback
\[
\xymatrix{
X_{hoequiv}\ar[d]\ar[r]& X_1=\Map(F(1),X)\ar[d]\\
\pi_0\Map(J,X)\ar[r]& \pi_0\Map(F(1),X)
}
\]
where the bottom map is induced by the map $F(1)\to J$.
\end{defi}

Clearly $X\mapsto X_{hoequiv}$ defines a functor from Segal fibrant simplicial spaces to spaces. Using our previous observation that the map $X_1\to \pi_0(X_1)$ is a fibration and the right properness of $\Sp$, we immediately see that $X\mapsto X_{hoequiv}$ sends levelwise weak equivalences to weak equivalences. Now we prove that this definition extends Rezk's definition of the space of homotopy equivalences.

\begin{prop}
The restriction of the functor $X\mapsto X_{hoequiv}$ to Segal spaces is naturally isomorphic to Rezk's space of homotopy equivalences defined in~\cite[section 5.7.]{rezkmodel}.
\end{prop}

\begin{proof}
Let $X$ be a Segal space. In this proof, we use the notation $X_{hoequiv}^R$ for the space of homotopy equivalences of $X$ defined by Rezk in~\cite[Section 5.7.]{rezkmodel}. Since $X_{hoequiv}^R$ is a set of components of $X_1$, we have
\[X_{hoequiv}^R=\pi_0(X_{hoequiv}^R)\times_{\pi_0(X_1)}X_1\] 

On the other hand, let us consider the following commutative diagram
\[
\xymatrix{
&X_{hoequiv}\ar[d]\ar[r]& X_1\ar[d]\\
\pi_0\Map(E,X)\ar[r]&\pi_0\Map(J,X)\ar[r]& \pi_0\Map(F(1),X)
}
\]
in which the square is the cartesian square defining $X_{hoequiv}$ (see definition~\ref{defi-X_{hoequiv}}). Since $X$ is a Segal space, the map $\pi_0\Map(E,X)\to\pi_0\Map(J,X)$ is an isomorphism, which implies that $X_{hoequiv}$ can be also defined as 
\[X_{hoequiv}=\pi_0\Map(E,X)\times_{\pi_0(X_1)} X_1\]

According to~\cite[Theorem 6.2.]{rezkmodel}, the map $\Map(E,X)\to X_1$ factors through $X_{hoequiv}^R$ and induces a weak equivalence $\Map(E,X)\to X_{hoequiv}^R$. In particular, it induces an isomorphism on $\pi_0$ which concludes the proof.
\end{proof}

The unique map $F(1)\to F(0)$ can be factored as $F(1)\to J \to F(0)$. If $X$ is a Segal fibrant simplicial space, we can apply $\Map(-,X)$, we find that the degeneracy $X_0\to X_1$ factors as 
\[X_0\to\Map(J,X)\to X_1\]
In particular, looking at the pullback square of definition~\ref{defi-X_{hoequiv}} we see that the degeneracy $X_0\to X_1$ factors through $X_{hoequiv}$. 

It is proved in~\cite[Theorem 7.2.]{rezkmodel} that the fibrant objects of $\cat{CSS}_{inj}$ are the Segal spaces such that the map $X_0\to X_{hoequiv}$ is a weak equivalence. These simplicial spaces are called complete Segal spaces.

Now, we give a characterization of the fibrant objects of $\cat{CSS}_{proj}$

\begin{prop}\label{prop-Rezk fibrant}
Let $X$ be a Segal fibrant simplicial space. The following conditions are equivalent.
\begin{enumerate}
\item $X$ is fibrant in $\cat{CSS}_{proj}$.
\item $X$ is local with respect to the unique map $E\to F(0)$.
\item The map $X_0\to X_{hoequiv}$ is a weak equivalence.
\item For any levelwise weak equivalence $X\to Y$ with $Y$ fibrant in $s\Sp_{inj}$, the simplicial space $Y$ is a complete Segal space.
\item There exists a levelwise weak equivalence $X\to Y$ such that $Y$ is a complete Segal space.
\end{enumerate}
\end{prop}

\begin{proof}
The equivalence of $(1)$ and $(2)$ and of $(4)$ and $(5)$ is formal and similar to the analogous result in the case of Segal spaces (see the proof of~\ref{prop-Segal fibrant}). The equivalence of $(5)$ and $(1)$ follows from~\cite[Proposition 3.6]{colemixing} since $\cat{CSS}_{proj}$ is the mixed model structure with the weak equivalences of $\cat{CSS}_{inj}$ and the cofibrations of $s\Sp_{proj}$.

$(5)\implies (3)$. Note that for a Segal fibrant simplicial space, satisfying $(3)$ is preserved under levelwise weak equivalences. Hence, if $X$ satisfies $(5)$, by~\cite[Theorem 7.2.]{rezkmodel}, $Y$ satisfies $(3)$ which implies that $X$ satisfies $(3)$.

$(3)\implies (4)$ Let $X\to Y$ be a levelwise weak equivalence with $Y$ fibrant in $s\Sp_{inj}$. By proposition~\ref{prop-Segal fibrant}, $Y$ is a Segal space. We have observed in the previous paragraph that satisfying $(3)$ is preserved under weak equivalences. Thus $Y$ satisfies $(3)$ which is precisely saying that $Y$ is a complete Segal space.  
\end{proof}

From now on, a fibrant object of $\cat{CSS}_{proj}$ will be called a Rezk fibrant simplicial space.

\subsection{The Dwyer-Kan equivalences.}

For $X$ a Segal fibrant simplicial space, the maps $d_0$ and $d_1$ from $X_1$ to $X_0$ induce maps $X_{hoequiv}\to X_0$.

\begin{defi}
For $X$ a Segal fibrant simplicial space, we define the set $\pi_0(X_0)/\sim$ to be the following coequalizer
\[(\pi_0(d_0),\pi_0(d_1)):\pi_0(X_{hoequiv})\rightrightarrows \pi_0(X_0)\] 
\end{defi}

\begin{defi}
We say that a map $f:X\to Y$ between Segal fibrant simplicial spaces is
\begin{itemize}
\item fully faithful if the square
\[
\xymatrix{
X_1\ar[d]_{(d_0,d_1)}\ar[r]^{f_1}& Y_1\ar[d]^{(d_0,d_1)}\\
X_0\times X_0\ar[r]_{f_0\times f_0}& Y_0\times Y_0
}
\]

is homotopy cartesian.

\item essentially surjective if the induced map $\pi_0(X_0)/\sim\to\pi_0(Y_0)/\sim$ is surjective.

\item a Dwyer-Kan equivalence if it is both fully faithful and essentially surjective.
\end{itemize}
\end{defi}

If $X$ is a Segal space and $x$ and $y$ are two $0$-simplices of $X_0$, we denote by $\map_X(x,y)$, the fiber of $X_1$ over $(x,y)$ along the map $(d_0,d_1):X_1\to X_0\times X_0$. Since $X$ is injectively fibrant, the map $(d_0,d_1)$ is a fibration which implies that $\map_X(x,y)$ is a Kan complex. It is proved in~\cite[Section 5]{rezkmodel} that these mapping spaces can be composed up to homotopy so that there is a category $\on{Ho}(X)$ whose objects are the $0$-simplices of $X_0$ and with
\[\on{Ho}(X)(x,y)=\pi_0\map_X(x,y)\]
Moreover, this category depends functorially on $X$.

Rezk in~\cite[Section 7.4.]{rezkmodel} defines a notion of Dwyer-Kan equivalence between Segal spaces. A map $f:X\to Y$ between Segal spaces is a Dwyer-Kan equivalence in Rezk's sense if 
\begin{itemize}
\item the induced map $\map_X(x,x')\to \map_Y(f(x),f(x'))$ is a weak equivalence for any pair of points $x$, $x'$ in $X_0$ .
\item the induced map $\on{Ho}(f):\on{Ho}(X)\to\on{Ho}(Y)$ is an equivalence of categories.
\end{itemize}

We want to prove that our definition of Dwyer-Kan equivalences coincides with Rezk's.

\begin{prop}
Let $f:X\to Y$ be a map between Segal spaces. Then 
\begin{enumerate}
\item the map $f$ is fully faithful if and only if, for any pair of points $(x,x')$ in $X_0$, the induced map
\[\map_X(x,x')\to\map_Y(f(x),f(x')))\]
is a weak equivalence.

\item the map $f$ is essentially surjective if and only if the induced map
\[\on{Ho}(f):\on{Ho}(X)\to\on{Ho}(Y)\]
is essentially surjective.

\item the map $f$ is a Dwyer-Kan equivalence if and only if it is a Dwyer-Kan equivalence in the sense of Rezk.
\end{enumerate}
\end{prop}

\begin{proof}
(1) By definition, the map $f$ is fully faithful if and only if the square
\[
\xymatrix{
X_1\ar[d]_{(d_0,d_1)}\ar[r]^{f_1}& Y_1\ar[d]^{(d_0,d_1)}\\
X_0\times X_0\ar[r]_{f_0\times f_0}& Y_0\times Y_0
}
\]
is homotopy cartesian. Since $X$ and $Y$ are injectively fibrant, the vertical maps are fibrations. Thus, by proposition~\ref{prop-characterization homotopy cartesian}, $f$ is fully faithful if and only if the map
\[\map_X(x,x')\to\map_Y(f(x),f(x'))\]
is a weak equivalence for any point $(x,x')$ in $X_0\times X_0$.

(2) It suffices to check that for any Segal space $X$, the set $\pi_0(X)/\sim$ is isomorphic to $\on{Ho}(X)/\cong$, the set of isomorphism classes of objects of $\on{Ho}(X)$. There is a surjective map $(X_0)_0\to\on{Ho}(X)/\cong$. We claim that this map factors through $\pi_0(X_0)$. 

Indeed, as displayed in~\cite[6.3]{rezkmodel}, the diagonal map $X_0\to X_0\times X_0$ factors through $X_{hoequiv}$. Taking $\pi_0$, we find that the map $\pi_0(X_{hoequiv})\to\pi_0(X_0)\times\pi_0(X_0)$ hits the diagonal. Let $x$ and $y$ be two points of $X_0$ that lie in the same path component. Let $u$ be a $1$ simplex of $X_0\times X_0$ connecting $(x,x)$ and $(x,y)$ (one can for instance take the product of the degenerate $1$-simplex at $x$ with any choice of $1$-simplex connecting $x$ and $y$). Let us consider the commutative diagram
\[\xymatrix{
\Delta[0]\ar[r]^{s_0(x)}\ar[d]^{d^0}& X_1\ar[d]^{(d_0,d_1)}\\
\Delta[1]\ar[r]_u& X_0\times X_0
}
\]
where the top maps classifies the point $s_0(x)$ in $X_1$. Since the map $(d_0,d_1)$ is a fibration, there is a lift in this diagram which implies that there is a $1$-simplex of $X_1$ connecting $s_0(x)$ and some point $h$ of $X_1$ such that $(d_0(h),d_1(h))=(x,y)$. Since $s_0(x)$ lies in $X_{hoequiv}$, so does $h$. This implies that $x$ and $y$ are isomorphic in $\on{Ho}(X)$. Therefore, we have a surjective map $P:\pi_0(X_0)\to \on{Ho}(X)/\cong$.

Let $v$ be a path component of $X_{hoequiv}$. Let $f$ be any point in $v$ and $x=d_0(f)$ and $y=d_1(f)$. The path component of $f$ in $\map_X(x,y)$ is an isomorphism $x\to y$ in $\on{Ho}(X)$ by~\cite[\S 5.7.]{rezkmodel}. Therefore, $x$ and $y$ get identified in $\on{Ho}(X)/\cong$. Thus the map $P$ induces a surjective map $Q:\pi_0(X_0)/\sim\to\on{Ho}(X)/\cong$.

Let us show that $Q$ is injective. Let $x$ and $y$ be two points of $X_0$ and $v$ be an isomorphism between them in $\on{Ho}(X)$. Let $f$ be any point in $\map_X(x,y)$ in the path component of $v$. Then $f$ seen as a point in $X_{hoequiv}$ identifies the path components of $x$ and $y$ in $\pi_0(X_0)/\sim$.

(3) If $f$ satisfies the equivalent conditions of (1), then the map $\on{Ho}(f)$ is fully faithful. Thus if $f$ is fully faithful and essentially surjective, then $f$ is a Dwyer-Kan equivalence in Rezk's sense. Conversely, if $f$ is a Dwyer-Kan equivalence in Rezk's sense, then $f$ is fully faithful and essentially surjective by (1) and (2).
\end{proof}

\begin{prop}
Let 
\[
\xymatrix{
X\ar[d]_i\ar[r]^f&Y\ar[d]^j\\
U\ar[r]_{g}& V
}
\]
be a commutative diagram between Segal fibrant simplicial spaces in which the vertical maps are levelwise weak equivalences. Then the map $f$ is a Dwyer-Kan equivalence if and only if the map $g$ is a Dwyer-Kan equivalence.
\end{prop}

\begin{proof}
We have an induced diagram
\[
\xymatrix{
\pi_0(X_0)/\sim\ar[d]\ar[r]& \pi_0(Y_0)/\sim\ar[d]\\
\pi_0(U_0)/\sim\ar[r]& \pi_0(V_0)/\sim
}
\]
The functor $\pi_0(-)/\sim$ sends levelwise weak equivalences to bijections, therefore, the two vertical maps are bijections. This informs us that $g$ is essentially surjective if and only if $f$ is essentially surjective.

We know that $j:Y\to V$ and $i:X\to U$ are levelwise weak equivalences. This implies that the square
\[
\xymatrix{
X_1\ar[d]\ar[r]& Y_1\ar[d]\\
X_0\times X_0\ar[r]& Y_0\times Y_0
}
\]
induced by $f$ maps to the square
\[
\xymatrix{
U_1\ar[d]\ar[r]& V_1\ar[d]\\
U_0\times U_0\ar[r]& V_0\times V_0
}
\]
induced by $g$ by a levelwise weak equivalence of squares. Note that this uses the classical fact that weak equivalences in $\Sp$ are stable under finite products. Thus the equivalence between the fully faithfulness of $f$ and $g$ follows from proposition~\ref{prop-invariance of homotopy cartesian squares}.
\end{proof}

We can now generalize~\cite[Theorem 7.7.]{rezkmodel} to Segal fibrant simplicial spaces.

\begin{prop}\label{prop-Dwyer Kan equivalences are Rezk equivalences}
Let $f:X\to Y$ be a map between Segal fibrant simplicial spaces. Then $f$ is a weak equivalence in $\cat{CSS}_{proj}$ if and only if it is a Dwyer-Kan equivalence.
\end{prop}

\begin{proof}
First observe that we can functorially replace a Segal fibrant simplicial space by a levelwise equivalent Segal space. Indeed, if $R$ is a fibrant replacement functor in $s\Sp_{inj}$, then $RX$ is levelwise weakly equivalent to $X$. Thus according to proposition~\ref{prop-Segal fibrant}, if $X$ is Segal fibrant, $RX$ is a Segal space. 

Now let us prove the proposition. Let $f:X\to Y$ be a map between Segal fibrant simplicial spaces. By the previous observation, we can embed $f$ into a commutative square
\[
\xymatrix{
X\ar[d]\ar[r]^f&Y\ar[d]\\
X'\ar[r]_{f'}& Y'
}
\]
in which $X'$ and $Y'$ are Segal spaces and the vertical maps are levelwise weak equivalences. By the two-out-of-three property for the Rezk equivalences, the map $f$ is a Rezk equivalence if and only if $f'$ is one. By the previous proposition, $f$ is a Dwyer-Kan equivalence if and only if $f'$ is one. But for $f'$, the two notions coincide by~\cite[Theorem 7.7.]{rezkmodel}.
\end{proof}

\section{Internal categories}

\subsection{Generalities.}

Let $P$ be a space, the category of $P$-graphs denoted $\cat{Graph}_P$ is the overcategory $\Sp_{/P\times P}$. This category has a (nonsymmetric) monoidal structure given by sending $(s_A,t_A):A\to P\times P$ and $(s_B,t_B):B\to P\times P$ to the fiber product $A\times _P B$
taken along the map $t_A$ and $s_B$.

\begin{rem}
If $(s_A,t_A):A\to P\times P$ is a $P$-graph we will always use the following convention. A fiber product $-\times_PA$ is taken along $s_A$ and a fiber product $A\times_P-$ is taken along $t_A$.
\end{rem}

\begin{defi}
The category of $P$-internal categories is the category of monoids in $\cat{Graph}_P$. We denote it by $\ICat_P$.
\end{defi}

If $u:P\to Q$ is a map of simplicial sets, we get a functor $u^*:\cat{Graph}_Q\to\cat{Graph}_P$ sending $A$ to the fiber product $P\times_QA\times_QP$. This functor is lax monoidal, therefore, it induces a functor
\[u^*:\ICat_Q\to\ICat_P\]

\begin{defi}
The category $\cat{ICat}$ is the Grothendieck construction of the pseudo-functor from $\Sp\op$ to large categories sending $P$ to $\ICat_P$.
\end{defi}

More concretely, $\ICat$ is the category whose objects are pairs $(P,M)$ of a simplicial set $P$ called the space of objects and a $P$-internal category $M$ called the space of arrows. The morphisms $(P,M)\to (Q,N)$ are the pairs $(u,f^u)$ where $u:P\to Q$ is a map in $\Sp$ and $f^u:M\to u^* N$ is a map in $\cat{ICat}_{P}$. 

The fact that fiber products are computed degreewise in $\Sp$ implies that there is an equivalence of categories $\ICat\to\Cat^{\Delta\op}$.

With this last description, it is obvious that the category $\ICat$ is locally presentable.

We use the notation $\Ob(C)$ to denote the space of objects of an internal category $C$ and $\Ar(C)$ to denote the space of arrows.

The category $\Sp$ is a full subcategory of $\ICat$ through the functor sending $K$ to $(K,K)$ where both source and target are the identity map. The internal categories in the image of this functor are called discrete. Similarly, if $C$ is an ordinary category, we can see it as an internal category whose space of objects and morphisms are discrete (i.e. are constant simplicial sets). This defines a fully faithful embedding $\Cat\to\ICat$. More generally, the category $\Cat_{\Delta}$ of simplicially enriched categories is the full subcategory of $\ICat$ spanned by the internal categories whose space of objects is discrete. We will make no difference in notations between a space and its image in $\ICat$ and between a (simplicially enriched) category and its image in $\ICat$ under these two functors (this convention will be modified in the last section in which we will study in details the inclusion functor $\Cat_{\Delta}\to\ICat$).

\begin{prop}
The category $\ICat$ is cartesian closed.
\end{prop}

\begin{proof}
If $C$ and $D$ are internal categories, we define an internal category $C^D$ with 
\[\Ob(C^D)_k=\ICat(D\times\Delta[k],C), \;\;\Ar(C^D)_k=\ICat(D\times[1]\times\Delta[k],C)\]
The internal category structure is left to the reader as well as the fact that there are natural isomorphisms $C^{D\times E}\cong (C^D)^E$.
\end{proof}

\subsection{The nerve functor.}

The main tool of this paper is the nerve functor $N:\ICat\to s\Sp$. It can be defined as the composite
\[N:\ICat\cong\Cat^{\Delta\op}\to \Sp^{\Delta\op}\to s\Sp\]
where the first map is the ordinary nerve functor applied degreewise and the functor $\Sp^{\Delta\op}\to s\Sp$ is the automorphism which swaps the two simplicial directions (we have chosen different notations to avoid confusion).

Concretely $N(C)$ is the simplicial space whose space of $n$-simplices is the $n$-fold fiber product
\[\Ar(C)\times_{\Ob(C)}\Ar(C)\times_{\Ob(C)}\times\ldots\times_{\Ob(C)}\Ar(C)\]

The nerve functor has a left adjoint $S:s\Sp\to \ICat$. The functor $S$ can be defined as the degreewise application of the left adjoint to the classical nerve functor $\Cat\to \Sp$ precomposed with the functor $s\Sp\to \Sp^{\Delta\op}$ that swaps the two simplicial directions. With this description, we see that the category of $k$-simplices of $S(X)$ is the quotient of the free category on the graph $(X_1)_k\rightrightarrows (X_0)_k$ where for any point $t$ in $(X_2)_k$, we impose the relation $d_2(t)\circ d_0(t)=d_1(t)$. Equivalently, the functor $S$ is the unique colimit preserving functor sending $F(p)\times\Delta[q]$ to $[p]\times\Delta[q]$.

Note that the functor $N$ is fully faithful. This implies that the  counit map $SN(C)\to C$ is an isomorphism for any internal category $C$.

\begin{prop}\label{prop-N preserves filtered colimits}
The functor $N:\ICat\to s\Sp$ preserves filtered colimits.
\end{prop}

\begin{proof}
The ordinary nerve functor $\Cat\to \Sp$ preserves filtered colimits because each of the categories $[n]$ is a compact object of $\cat{Cat}$. The functor $N$ is the ordinary nerve applied in each degree. Since colimits in $\ICat$ and $s\Sp$ are computed degreewise, we are done.
\end{proof}

\subsection{Mapping spaces.}

Let $C$ and $D$ be internal categories. We use the notation $\Map(C,D)$ for the mapping space $\Map(NC,ND)$ in the category of simplicial spaces. This mapping space has as $n$ simplices the set of maps of bisimplicial sets $NC\times\Delta[n]\to ND$.

The simplicial space $\Delta[n]$ can be identified with the nerve of the discrete internal category $\Delta[n]$ (i.e. the internal category whose space of objects and space of morphism are both $\Delta[n]$). Therefore, the $n$ simplices of $\Map(C,D)$ are equivalently the maps of internal categories $C\times \Delta[n]\to D$.

Hence we see that $\Map(C,D)$ is the space $\Ob(D^C)$. It is also clear from this description that the functor $\Map(-,-)$ from $\ICat\op\times\ICat$ to $\Sp$ preserves limits in both variables.

\begin{rem}\label{rem-N is simplicially enriched}
By definition of the mapping space in $\ICat$, the nerve functor is a simplicially enriched functor. Moreover, the nerve also preserves cotensors by simplicial set. That is, if $C$ is an internal category and $K$ is a simplicial set, then there is a natural isomorphism $N(C^K)\cong N(C)^K$. Thus by~\cite[Theorem 4.85.]{kellybasic}, the adjunction $(S,N)$ is a simplicial adjunction.
\end{rem}

\section{A key lemma}

As usual, the main difficulty when one tries to transfer a model structure along a right adjoint is that the right adjoint does not preserve pushouts. The case of the nerve functor is no exception. However, in this section, we prove that certain very particular pushouts in $\ICat$ are preserved by the nerve functor.

The functor $\cat{Set}\to\Cat$ sending a set to the discrete category on that set has a left adjoint $\pi_0$. Concretely $\pi_0(C)$ is the quotient of the set $\Ob(C)$ by the smallest equivalence relation containing the pairs $(c,d)$ such that at least one of $C(c,d)$ or $C(d,c)$ is non empty. We say that a category $C$ is connected if $\pi_0(C)$ consists of a single element. Note that if $B$ is connected, then the set of functors $B\to C\sqcup D$ splits as $\Cat(B,C)\sqcup \Cat(B,D)$.

\begin{lemm}\label{lemm-key lemma}
Let $A$ be an object of $\cat{Cat}$ and $i:K\to L$ be a monomorphism in $\Sp$. Let
\[
\xymatrix{
K\times A\ar[r]\ar[d]_{i\times \id}&C\ar[d]^f\\
L\times A\ar[r] & D
}
\]
be a pushout diagram in $\ICat$. Then for each $B\in\cat{Cat}$ that is connected, the induced square
\[
\xymatrix{
\Map(B,K\times A)\ar[r]\ar[d]&\Map(B,C)\ar[d]\\
\Map(B,L\times A)\ar[r] & \Map(B,D)
}
\]
is a pushout diagram in $\Sp$.
\end{lemm}

\begin{proof}
It suffices to prove that for each $k$, the square
\[
\xymatrix{
\Map(B,K\times A)_k\ar[r]\ar[d]&\Map(B,C)_k\ar[d]\\
\Map(B,L\times A)_k\ar[r] & \Map(B,D)_k
}
\]
is a pushout square of sets. Equivalently, it suffices to prove that for each $k$, the square in $\cat{Set}$
\[
\xymatrix{
\ICat(B\times\Delta[k],K\times A)\ar[r]\ar[d]&\ICat(B\times\Delta[k],C)\ar[d]\\
\ICat(B\times\Delta[k],L\times A)\ar[r] & \ICat(B\times\Delta[k],D)
}
\]
is a pushout square. This is equivalent to proving that 
\begin{equation}\label{square}
\xymatrix{
\Cat(B,K_k\times A)\ar[r]\ar[d]&\Cat(B,C_k)\ar[d]\\
\Cat(B,L_k\times A)\ar[r] & \Cat(B,D_k)
}
\end{equation}
is a pushout square, where now each corner is just the set of functors between ordinary categories.

Colimits in $\ICat$ are computed degreewise. Hence, for each $k$, we have a pushout diagram in $\Cat$
\[
\xymatrix{
K_k\times A\ar[r]\ar[d]_{i_k\times \id}&C_k\ar[d]^{f_k}\\
L_k\times A\ar[r] & D_k
}
\]
Let us denote by $Z_k$ the set $L_k-K_k$. Then the category $D_k$ is isomorphic to $C_k\sqcup Z_k\times A$ and the map $f_k$ is the obvious inclusion. 

Since the category $B$ is connected, there is an isomorphism
\[\Cat(B,D_k)\cong\Cat(B,C_k)\sqcup\Cat(B, Z_k\times A)\]
and an isomorphism $\Cat(B,S)\cong S$ for each set $S$. Hence we have
\[\Cat(B,D_k)=\Cat(B,C_k)\sqcup(\Cat(B, A)\times Z_k)\] 

On the other hand, we can compute
\[\Cat(B,C_k)\sqcup^{\Cat(B,K_k\times A)}\Cat(B,L_k\times A)\]
By connectedness of $B$, this coincides with 
\[\Cat(B,C_k)\sqcup^{\Cat(B,A)\times K_k}\Cat(B,A)\times L_k\]
which is clearly isomorphic to $\Cat(B,C_k)\sqcup(\Cat(B,A)\times Z_k)$ which finishes the proof that (\ref{square}) is a pushout square.
\end{proof}

\begin{coro}\label{coro-key lemma}
We keep the notations and hypothesis of the previous lemma. The square
\[
\xymatrix{
N(K\times A)\ar[r]\ar[d]_{N(i\times \id)}&NC\ar[d]^{Nf}\\
N(L\times A)\ar[r] & ND
}
\]
is a pushout square in $s\Sp$
\end{coro}

\begin{proof}
It suffices to check it in each degree. But the category $[n]$ is connected for all $n$, hence according to the previous proposition, the square
\[
\xymatrix{
N_n(K\times A)\ar[r]\ar[d]_{N_n(i\times \id)}&N_nC\ar[d]^{N_nf}\\
N_n(L\times A)\ar[r] & N_nD
}
\]
is a pushout square in $\Sp$.
\end{proof}

Using this fact, we have the following proposition which gives a necessary condition for a simplicial space to be cofibrant in the projective model structure.

\begin{prop}\label{prop-characterization of cofibrants}
Let $X$ be a cofibrant simplicial space. Then the unit map $X\to NSX$ is an isomorphism.
\end{prop}

\begin{proof}
First we notice that the unit map $X\to NSX$ is an isomorphism if and only if $X\cong NC$ for some $C$ in $\ICat$. Indeed if $X\cong NC$, then $NSX\cong NSNC\cong NC\cong X$ by fully faithfulness of $N$. We say that $X$ is a nerve if $X\to NSX$ is an isomorphism. The proof is now divided in a few steps.

(1) If $X$ is a nerve and $F(n)\times K\to X$ is any map, then for any monomorphism $K\to L$ in $\Sp$, the pushout of
\[\xymatrix{
F(n)\times K\ar[d]\ar[r]& X\\
F(n)\times L& 
}
\]
is a nerve. Indeed by the previous proposition, the pushout is the nerve of the pushout of the following diagram in $\ICat$:
\[\xymatrix{
[n]\times K\ar[d]\ar[r]& SX\\
[n]\times L& 
}
\]

(2) If $X=\on{colim}_{i\in I} X_i$ is a filtered colimit of nerves, then $X$ is a nerve. Indeed, if for all $i$, the map $X_i\to NSX_i$ is an isomorphism, then so is $X\to NSX$ since $N$ and $S$ both preserve filtered colimits.

(3) Let $\alpha$ be some ordinal. let $X_0\to X_1\to \ldots\to Y=\on{colim}_{\beta<\alpha} X_\beta$ be a transfinite composition of maps in $s\Sp$ such that $X_0$ is a nerve and each map in the transfinite composition is a pushout of a map of the form $F(n)\times K\to F(n)\times L$ for some integer $n$ and some monomorphism $K\to L$. Then we claim that $Y$ is a nerve. This is a transfinite induction argument. If $X_\beta$ is a nerve for some ordinal $\beta<\alpha$, then $X_{\beta+1}$ is a nerve by (1). If $\beta$ is a limit ordinal and $X_\gamma$ is a nerve for all $\gamma<\beta$, then $X_\beta=\on{colim}_{\gamma<\beta}X_\gamma$ is a nerve by (2).

(4) If $X$ is a nerve, then any retract of $X$ is a nerve. Indeed, if $Y\to X\to Y$ is a retract, then the map $Y\to NSY$ is a retract of $X\to NSX$. Therefore, if $X\to NSX$ is an isomorphism, so is $Y\to NSY$.

(5) To conclude the proof it suffices to recall that if $X$ is cofibrant, then $X$ is a retract of some cell complex $Y$ in $s\Sp_{proj}$. And by definition of a cell complex, the map $\varnothing \to Y$ is a transfinite composition of pushouts of maps of the form $F(n)\times K\to F(n)\times L$ with $K\to L$ a monomorphism in $\Sp$. Since $\varnothing$ is a nerve, we are done.
\end{proof}

\section{The model structure}

We say that a map in $s\Sp$ is a levelwise (resp. Segal, resp. Rezk) weak equivalence if it is a weak equivalence in $s\Sp_{proj}$ (resp. $\cat{SS}_{proj}$, resp. $\cat{CSS}_{proj}$).
We say that a map in $\ICat$ is a levelwise (resp. Segal, resp. Rezk) weak equivalence if its nerve is a levelwise (resp. Segal, resp. Rezk) weak equivalence of simplicial spaces. In this section, we construct three model structures on $\ICat$ whose weak equivalences are respectively the levelwise, Segal and Rezk weak equivalences.

\subsection{The levelwise model structure.}\label{Icat projective}

Let $I_\Sp$ and $J_\Sp$ be a set of generating cofibrations and trivial cofibrations in $\Sp$. The projective model structure on $s\Sp$ admits the maps $f\times F(n)$ with $f$ in $I_\Sp$ (resp. $f\in J_\Sp$) and $n\in\mathbb{Z}_{\geq 0}$  as generating cofibrations (resp. generating trivial cofibrations). We denote those sets by $I$ and $J$.

We can now prove the following:

\begin{theo}
There is a model structure on $\ICat$ whose weak equivalences are the levelwise weak equivalences and whose fibrations are the maps whose nerve is a fibration in $s\Sp_{proj}$. Its cofibrations are the $SI$-cofibrations and its trivial cofibrations are the $SJ$-cofibrations. Moreover the functor $N$ preserves cofibrations.
\end{theo}

\begin{proof}
We apply theorem~\ref{theo-transferred model structure}. We already know that $N$ preserves filtered colimits by proposition~\ref{prop-N preserves filtered colimits}. 

We need to check that $N$ of a pushout of a map in $SI$ is an $I$-cofibration. Let $i:K\times F(n)\to L\times F(n)$ be a map in $I$. Then $Si$ can be identified with $K\times[n]\to L\times[n]$. Let us consider a pushout square
\[
\xymatrix{
K\times[n]\ar[d]_{Si}\ar[r]^-u&C\ar[d]^f\\
L\times[n]\ar[r]& D}
\]
According to corollary~\ref{coro-key lemma}, the map $N(f)$ is the pushout of $NS(i)=i$ along $N(u)$. In particular, it is an $I$-cofibration. Similarly $N$ of a pushout of a map of $J$ is a $J$-cofibration.
\end{proof}

We denote by $\ICat^{LW}$ the category of internal categories equipped with this model structure. Note that this model category is cofibrantly generated. Since $\ICat$ is locally presentable, $\ICat^{LW}$ is combinatorial.

\begin{prop}
The model category $\ICat^{LW}$ is proper.
\end{prop}

\begin{proof}
The right properness follows directly from the right properness of $s\Sp_{proj}$, using the fact that the functor $N$ preserves pullbacks, fibrations and preserves and reflects weak equivalences. 

For the left properness, first notice that the weak equivalences in $\ICat^{LW}$ are stable under filtered colimits. Indeed, weak equivalences are preserved and reflected by the nerve functor to $s\Sp$ which is a filtered colimit preserving functor. The levelwise weak equivalences in $s\Sp$ are stable under filtered colimits because colimits are computed levelwise and the same is true in $\Sp$.

Because of this observation, in order to show that $\ICat$ is left proper, it suffices to prove that for any generating cofibration $K\times[n]\to L\times[n]$ and any weak equivalence $v:C\to D$ in $\ICat^{LW}$ fitting in a diagram
\[
\xymatrix{
K\times[n]\ar[r]\ar[d]&C\ar[r]^v\ar[d]& D\ar[d]\\
L\times[n]\ar[r]&E\ar[r]^w&F}
\]
where both squares are pushouts, the map $w$ is a weak equivalence. We can hit this diagram with $N$ and we get a diagram in $s\Sp$
\[
\xymatrix{
N(K\times[n])\ar[r]\ar[d]&NC\ar[r]^{Nv}\ar[d]& ND\ar[d]\\
N(L\times[n])\ar[r]&NE\ar[r]^{Nw}&NF}
\]
Because of corollary~\ref{coro-key lemma}, the leftmost square and the total square are pushouts. This implies that the rightmost square is a pushout square. But now the result follows directly from the left properness of $s\Sp_{proj}$ and the fact that $N$ preserves cofibrations.
\end{proof}

\begin{prop}\label{prop-simplicial model structure}
The functor $\Map:(\ICat^{LW})\op\times\ICat^{LW}\to \Sp$ makes $\ICat^{LW}$ into a simplicial model category.
\end{prop}

\begin{proof}
This follows from proposition~\ref{prop-transferred is simplicial} and the fact that the adjunction $(S,N)$ is simplicial (see remark~\ref{rem-N is simplicially enriched}).
\end{proof}

\begin{prop}
The Quillen adjunction 
\[S:s\Sp_{proj}\leftrightarrows \ICat^{LW}:N\]
is a Quillen equivalence.
\end{prop}

\begin{proof}
Let $X\in s\Sp_{proj}$ be cofibrant and $C$ in $\ICat^{LW}$ be fibrant. Let $f:SX\to C$ be a map. Since the functor $N$ preserves and reflects weak equivalences, $f$ is a weak equivalence if and only of $N(f):NSX\to NC$ is a weak equivalence. But since $X$ is cofibrant, the unit map $X\to NSX$ is an isomorphism by proposition~\ref{prop-characterization of cofibrants}. Therefore, $f$ is a weak equivalence if and only if its adjoint $g:X\to NC$ is a weak equivalence.
\end{proof}

\begin{rem}
The transfer of the model structure along the map $\ICat\to s\Sp_{proj}$ is analogous to~\cite[Theorem 7.13]{fioremodel}. In that paper, the authors transfer the projective Thomason model structure on simplicial objects in $\Cat$ to a model structure on the category of double categories (i.e. internal categories in categories). In particular,~\cite[Theorem 10.7]{fioremodel} should be compared to corollary~\ref{coro-key lemma}.
\end{rem}

\subsection{The Segal model structure.}\label{Icat Segal}

\begin{defi}
The category $\ICat^S$ is the left Bousfield localization of $\ICat^{LW}$ with respect to the maps $\mathbb{L}SG(n)\to \mathbb{L}SF(n)$.
\end{defi}

This Bousfield localization exists by~\ref{theo-existence Bousfield localization} since $\ICat^{LW}$ is left proper and combinatorial and is simplicial since the generating cofibrations of $\ICat^{LW}$ have cofibrant source.

By proposition~\ref{prop-Bousfield localization}, we have a Quillen equivalence
\[S:\cat{SS}_{proj}\leftrightarrows \ICat^S:N\]
in which the functor $N$ preserves and reflects weak equivalences.

\subsection{The Rezk model structure.}\label{Icat Rezk}

Recall that $I[1]$ denotes the groupoid completion of the category $[1]$.

We can now state the main theorem of the paper:

\begin{theo}\label{theo-main}
There is a left proper and simplicial model structure $\ICat$ on the category of internal categories in simplicial sets whose weak equivalences are the Rezk equivalences, and whose cofibrations are the cofibrations of $\ICat^{LW}$. Moreover, the adjunction
\[S:\cat{CSS}_{proj}\leftrightarrows \ICat:N\]
is a Quillen equivalence and the functor $N$ preserves and reflects weak equivalences.
\end{theo}

\begin{proof}
The model category $\ICat$ is defined to be the left Bousfield localization of $\ICat^S$ with respect to the unique map $\mathbb{L}SNI[1]\to [0]$. The existence of the model structure and the fact that it is simplicial and left proper follows from theorem~\ref{theo-existence Bousfield localization}. The equivalence with $\cat{CSS}_{proj}$ and the fact that $N$ preserves and reflects weak equivalences follows from~\ref{prop-Bousfield localization}.
\end{proof} 

\begin{rem}
In~\cite{shulmaninternal}, Jeremy Hahn gave another proof of the equivalence of relative categories $N:\ICat\to\cat{CSS}_{proj}$ relying on~\cite{barwickrelative}. 
\end{rem}

\subsection{The fibrant objects.}

We can characterize the fibrant internal categories.

\begin{prop}\label{prop-N preserves fibrant objects}

The fibrant objects in $\ICat^S$ and $\ICat$ are the internal categories whose nerve is fibrant in $\cat{SS}_{proj}$ and $\cat{CSS}_{proj}$ respectively.
\end{prop}

\begin{proof}
We do the case of $\ICat^S$. The other case is entirely analogous.

Let $C$ be an internal category. Then $C$ is fibrant in $\ICat^S$ if and only if it is fibrant in $\ICat^{LW}$ and for each $n$ the map
\[\Map(SQF(n),C)\to\Map(SQG(n),C)\]
induced by the map $G(n)\to F(n)$ is a weak equivalence. Since $(S,N)$ is a Quillen adjunction, this is equivalent to asking for $NC$ to be projectively fibrant and the maps
\[\Map(QF(n),NC)\to\Map(QG(n),NC)\]
to be weak equivalences, which is equivalent to $NC$ being fibrant in $\cat{SS}_{proj}$.
\end{proof}

If we unwrap this proposition, using proposition~\ref{prop-Segal fibrant}, we see that an internal category $C$ is fibrant in $\ICat^S$ if it is fibrant in $\ICat^{LW}$ and the commutative squares
\[\xymatrix{
\Ar(C)\times_{\Ob(C)}N_n(C)\ar[r]\ar[d]&N_n(C)\ar[d]\\
\Ar(C)\ar[r]&\Ob(C)
}
\]
are homotopy cartesian. By construction, these square are strictly cartesian, but this does not imply that they are homotopy cartesian.

For $C$ an internal category whose nerve is Segal fibrant, we denote by $C_{hoequiv}$ the space $N(C)_{hoequiv}$. Using proposition~\ref{prop-Rezk fibrant}, we see that an internal category $C$ is fibrant in $\cat{ICat}$ if it is fibrant in $\ICat^S$ and the map 
\[\Ob(C)\to C_{hoequiv}\]
sending an object to the identity at that object is a weak equivalence.

We will use the terminology ``Segal fibrant'' (resp. ``Rezk fibrant'') internal category to refer to a fibrant object of $\ICat^S$ (resp. $\ICat$).

\subsection{The Dwyer-Kan equivalences.}

\begin{defi}
We say that a map $f:C\to D$ between two Segal fibrant internal categories is fully faithful (resp. essentially surjective, resp. a Dwyer-Kan equivalence) if $N(f)$ is fully faithful (resp. essentially surjective, resp. a Dwyer-Kan equivalence).
\end{defi}

We then have the following immediate proposition.

\begin{prop}\label{prop-Dwyer Kan equivalences are Rezk equivalences ICat}
A map $f:C\to D$ between Segal fibrant internal categories is a Rezk equivalence if and only if it is a Dwyer-Kan equivalence.
\end{prop}

\begin{proof}
This follows from proposition~\ref{prop-Dwyer Kan equivalences are Rezk equivalences} and the fact that a map is a Rezk equivalence if and only if its nerve is one.
\end{proof}

\subsection{The category of strongly Segal internal categories.}

In this subsection, we define a particularly nice class of objects of $\ICat$ that we call strongly Segal internal categories. In particular, we will see in the last section that the strongly Segal internal categories have well-behaved presheaf categories.

\begin{defi}
A strongly Segal internal category is an internal category $C$ such that $\Ob(C)$ is fibrant and such that the source and target maps $\Ar(C)\to\Ob(C)$ are fibrations in $\Sp$.
\end{defi}

\begin{prop}
A strongly Segal internal category is fibrant in $\ICat^S$
\end{prop}

\begin{proof}
Let $C$ be a strongly Segal internal category. By proposition~\ref{prop-N preserves fibrant objects}, it suffices to check that $NC$ is fibrant in $\cat{SS}_{proj}$. 

We write $N_nC$ for the degree $n$ space of $NC$. We denote by $s$ the map $N_nC\to N_0C$ which is the composite of the leftmost projection $N_nC\to N_1C=\Ar(C)$ with the source $\Ar(C)\to \Ob(C)=N_0(C)$. Similarly $t:N_nC\to N_0C$ is the rightmost projection composed with the target $Ar(C)\to \Ob(C)$.

Let us prove that the maps $s$ and $t$ from $N_nC$ to $N_0C$ are fibrations. We do it by induction on $n$. This is by assumption true for $n=1$. Let $n$ be an integer. We have a commutative diagram in which the square is cartesian
\[\xymatrix{
N_nC\ar[d]_p\ar[r]_-q&N_{n-1}C\ar[d]_s\ar[r]_t&N_0C\\
N_1C\ar[d]_s\ar[r]_t&N_0C& \\
N_0C& & }
\]

By the induction hypothesis all maps but possibly $p$ and $q$ are fibrations. Since the square is cartesian, $p$ and $q$ must be fibrations as well. Therefore, since $N_0C$ is fibrant, $NC$ is levelwise fibrant. We also see from this diagram that the square
\[\xymatrix{
N_nC\ar[d]_p\ar[r]_-q&N_{n-1}C\ar[d]_s\\
N_1C\ar[r]_t&N_0C
}
\]
is homotopy cartesian. This implies by proposition~\ref{prop-Segal fibrant} that $NC$ is fibrant in $\cat{SS}_{proj}$.
\end{proof}

\section{Categories of internal functors}

In this section, we study the category of internal presheaves  over a strongly Segal internal category $C$. We put a model structure on this category that reduces to the projective model structure in the case where $C$ is a simplicial category.

\subsection{The over category model structure.}

The overcategory model structure on $\Sp_{/P}$ is the category whose cofibrations, fibrations and weak equivalences are the maps that are sent to a cofibration, fibration or weak equivalence by the forgetful functor $\Sp_{/P}\to \Sp$. A set of generating (trivial) cofibrations is obtained by taking all the commutative triangles
\[
\xymatrix{
K\ar[rd]\ar[rr]& & L\ar[dl]\\
 & P&
}
\]
with $K\to L$ a generating (trivial) cofibration in $\Sp$.

Let $f:(X,p_X)\to (Y,p_Y)$ be a map in $\Sp_{/P}$. We say that $f$ is a weak equivalence over $p\in P$ if the induced map $\on{hofiber}_pX\to \on{hofiber}_pY$ is a weak equivalence in $\Sp$.

\begin{lemm}
Let $P$ be a fibrant space and $p$ be a point of $P$. If a map $f:(X,p_X)\to (Y,p_Y)$ is a weak equivalence over $p$, it is a weak equivalence over any point in the path component of $P$.
\end{lemm}

\begin{proof}
Since homotopy fibers are invariant under weak equivalences of spaces over $P$, we can assume without loss of generality that $X$ and $Y$ are fibrant in $\Sp_{/P}$, i.e. that the structure maps $p_X$ and $p_Y$ are fibrations.

Let $q$ be a point in $P$ in the path component of $p$. Let $u$ be a $1$-simplex of $P$ whose faces are $p$ and $q$. Then, we can consider the following commutative diagram
\[
\xymatrix{
\Delta[0]\ar[d]_p\ar[r]&\Delta[1]\ar[d]_u&\Delta[0]\ar[d]^q\ar[l]\\
P\ar[r]&P&P\ar[l]
}
\]

For a space $K\to P$ in $\Sp_{/P}$, we denote by $K_p$ (resp. $K_q$, resp. $K_u$) the fiber product $K\times_{P}\Delta[0]$ taken along $p:\Delta[0]\to P$ (resp. $K\times_P\Delta[0]$ taken along $q$, resp. $K\times_P\Delta[1]$ taken along $u$). 

These three constructions are functorial in $K$. In particular, we have a commutative diagram
\[
\xymatrix{
X_p\ar[r]\ar[d]_{f_p}&X_u\ar[d]_{f_u}&X_q\ar[l]\ar[d]_{f_q}\\
Y_p\ar[r]&Y_u&Y_q\ar[l]
}
\]
All the horizontal maps in this diagram are weak equivalences because of our fibrancy assumption on $X$ and $Y$ and the right properness of $\Sp$. By assumption, the map $f_p$ is a weak equivalence. By the two out of three property, this implies that $f_q$ is a weak equivalence.
\end{proof}

We say that a map $X\to Y$ in $\Sp_{/P}$ is a weak equivalence over a path component of $P$ if it is a weak equivalence over one point in that path component. The previous lemma tells us that it is equivalent to $f$ being a weak equivalence over all the points of that path component.

\begin{prop}\label{prop-weak equivalences over a fixed space}
Let $P$ be a fibrant space. Then a map $f:(X,p_X)\to (Y,p_Y)$ is a weak equivalence in $\Sp_{/P}$ if and only if it is a weak equivalence over each path component of $P$. 
\end{prop}

\begin{proof}
Clearly, if $f$ is a weak equivalence it is a weak equivalence over each point of $P$. 

Conversely, if $f$ is a weak equivalence over each path component of $P$, according to the previous lemma, it is a weak equivalence over each point of $P$. Thus, by proposition~\ref{prop-characterization homotopy cartesian}, the following square is homotopy cartesian.
\[\xymatrix{
X\ar[r]^f\ar[d]_{p_X}& Y\ar[d]^{p_Y}\\
P\ar[r]_{\id_P}&P
}
\]
Since the bottom horizontal map is a fibration, this implies that $f$ is a weak equivalence.
\end{proof}

\subsection{Internal functors.}

Let $P$ be a space. There is a functor
\[\Sp_{/P}\times\Sp_{/P\times P}\to \Sp_{/P}\]
sending the pair $(X\xrightarrow{p} P, M\goto{(s,t)} P\times P)$ to the fiber product $X\times_PM$ taken along $p$ and $s$ equipped with the map $X\times_PM\to P$ induced by $t$. 

It is easy to verify that this functor makes $\Sp_{/P}$ into a category right tensored over the monoidal category $\Sp_{/P\times P}$. Hence, we can talk about a right module in $\Sp_{/P}$ over a monoid in $\Sp_{/P\times P}$ i.e. an internal category with space of objects $P$.

\begin{defi}
Let $C$ be an internal category. The category $\Sp^C$ is the category of right $\Ar(C)$-modules in $\Sp_{/\Ob(C)}$.
\end{defi}

More explicitly, an object of $\Sp^C$ is a space $F$ equipped with a map $F\to \Ob(C)$ together with an action map $F\times_{\Ob(C)} \Ar(C)\to F$ which is associative and unital. 

If $F$ is an object of $\Sp^C$, and $c$ is a point of $\Ob(C)$, then the fiber $F_c$ of $F$ over $c$ should be interpreted as the value of the functor $F$ at $c$. We note that if $u$ is a point in $\Ar(C)$ whose source is $c$ and target is $d$, we get a map $F_c\to F_d$. In particular, it is straightforward to check that if $C$ has a discrete space of objects (that is if $C$ is a simplicially enriched category), the category $\Sp^C$ is equivalent to the category of simplicial functors $C\to \Sp$.

\begin{prop}
Let $C$ be an internal category with space of objects $P$ for which the source map is a fibration. There is a simplicial model structure on $\Sp^C$ whose fibrations and weak equivalences are the maps that are sent to a fibration or weak equivalence by the forgetful functor $U:\Sp^C\to \Sp_{/P}$. Moreover, the forgetful functor $\Sp^C\to \Sp_{/P}$ preserves cofibrations.
\end{prop}

\begin{proof}
The left adjoint to the forgetful functor $U:\Sp^C\to \Sp_{/P}$ sends $Y\to P$ to $Y\times_P \Ar(C)$. We apply theorem~\ref{theo-transferred model structure}. Let $K\to L$ be a cofibration in $\Sp_{/P}$, then $K\times_P\Ar(C)\to L\times_P \Ar(C)$ is a cofibration in $\Sp_{/P}$. Indeed, it suffices to check that the underlying map in $\Sp$ is a monomorphism which is trivial. If $K\to L$ is a trivial cofibration then $K\times_P\Ar(C)\to L\times_P \Ar(C)$ is a cofibration by what we have just said and is a weak equivalence because the source map $\Ar(C)\to P$ is a fibration. Since the functor $U$ preserves all colimits, we are done.

The fact that $\cat{S}^C$ is a simplicial model category follows from proposition~\ref{prop-transferred is simplicial} and the fact that the adjunction $\Sp_{/P}\leftrightarrows\Sp^C$ is simplicial.
\end{proof}

\begin{rem}                    
Note that if $C$ is a simplicial category, then $\Sp^C$ coincides with the category of simplicial functors $C\to \Sp$. Moreover, if $C$ is fibrant as a simplicial category, then the source and target maps $\Ar(C)\to P$ are fibrations and the model structure we get on $\Sp^C$ is exactly the projective model structure.
\end{rem}

\subsection{A cofibrant replacement functor on $\Sp^C$.}

In the following, we shall always write $P$ for $\Ob(C)$. For $F$ an object of $\Sp^C$, we have the bar resolution
\[F\times_P\Ar(C)\leftleftarrows F\times_P\Ar(C)\times_P\Ar(C) \mathrel{\substack{\textstyle\leftarrow\\[-0.6ex]
                      \textstyle\leftarrow \\[-0.6ex]
                      \textstyle\leftarrow}} F\times_P\Ar(C)\times_P\Ar(C)\times_P\Ar(C)\ldots\]
induced by the adjunction $\Sp_{/P}\leftrightarrows \Sp^C$.

We will prove that its realization is a cofibrant replacement in $\Sp^C$. Before doing so, we recall a little bit of terminology about simplicial objects. We denote by $\Delta_s$ the category whose objects are the same as the objects of $\Delta$ but where we only keep the maps that are compositions of degeneracies. Let $\mathbb{N}$ denote the poset of nonnegative integers. There is a functor $\Delta_s\op\to\mathbb{N}$ sending the object $[n]$ to $n$.  This makes $\Delta_s\op$ into a direct category (see~\cite[Definition 5.1.1.]{hoveymodel}).

Let $\cat{M}$ be a model category and $X:\Delta_s\op\to\cat{M}$ be a functor. For any nonnegative integer $r$, we define as in~\cite[Definition 5.1.2.]{hoveymodel} the $r$-th latching object of $X$ denoted $L_rX$ by the colimit
\[L_rX=\on{colim}_{[n]\to [r], n\neq r}X_n\]

Note that this definition extends the definition of the $r$-th latching object of a simplicial object. More precisely, the $r$-th latching object of a simplicial object can be computed as the $r$-th latching object of its restriction along the inclusion $\Delta_s\op\to\Delta\op$.

We say that an object $X$ of $\on{Fun}(\Delta_s\op,\cat{M})$ is Reedy cofibrant if for each $r\geq 0$, the map $L_rX\to X_r$ is a cofibration. According to our previous observation, a simplicial diagram in $\cat{M}$ is Reedy cofibrant if and only if its restriction to $\Delta_s\op$ is Reedy cofibrant.
                     
\begin{prop}\label{prop-bar is cofibrant}
The bar resolution is Reedy cofibrant in $\Sp^C$
\end{prop}

\begin{proof}
As explained above, it is enough to prove that the restriction of the bar construction to $\Delta_s\op$
\begin{equation}\label{first diagram}
F\times_P\Ar(C)\rightarrow F\times_P\Ar(C)\times_P\Ar(C)\rightrightarrows F\times_P\Ar(C)\times_P\Ar(C)\times_P\Ar(C)\ldots
\end{equation}
is Reedy cofibrant. 

The functor $Y\mapsto Y\times_P\Ar(C)$ is a left Quillen functor from $\Sp_{/P}$ to $\Sp^C$. In particular, it preserves cofibrations and colimits. This immediately imply that the induced functor
\[-\times_P\Ar(C):\on{Fun}(\Delta_s\op,\Sp_{/P})\to\on{Fun}(\Delta_s\op,\Sp^C)\]
preserves Reedy cofibrant objects. The diagram (\ref{first diagram}) arises as $-\times_P\Ar(C)$ applied to the following diagram $\Delta_s\op\to \Sp_{/P}$:
\begin{equation}\label{second diagram}
F\rightarrow \Ar(C)\times_PF\rightrightarrows \Ar(C)\times_P\Ar(C)\times_PF\ldots
\end{equation}
Therefore, it suffices to prove that the latter diagram is Reedy cofibrant. 

The functor $X\mapsto X\times_PF$ from $\Sp_{/P\times P}$ to $\Sp_{/P}$ preserves cofibrations and colimits (note that it is not a left Quillen functor since it does not preserve weak equivalences in general). Therefore, the induced functor
\[-\times_PF:\on{Fun}(\Delta_s\op,\Sp_{/P\times P})\to\on{Fun}(\Delta_s\op,\Sp_{/P})\]
preserves Reedy cofibrant objects. In particular, we see that the diagram (\ref{second diagram}) is Reedy cofibrant if the diagram
\[P\to \Ar(C)\rightrightarrows \Ar(C)\times_P\Ar(C)\ldots\]
is Reedy cofibrant in $\Sp_{/P}$. Since colimits in $\Sp_{/P}$ are created by the forgetful functor $\Sp_{/P}\to \Sp$, it suffices to check that the underlying diagram $\Delta_s\op\to\Sp$ is Reedy cofibrant. But this diagram is really the restriction along $\Delta_s\op\to\Delta\op$ of the nerve of $C$. Therefore, it suffices to check that the nerve of $C$ is Reedy cofibrant but this last fact is true since by~\cite[Theorem 15.8.7.]{hirschhornmodel}, any simplicial space is Reedy cofibrant.
\end{proof}

The bar resolution is augmented over the constant simplicial object with value $F$ via the structure map $F\times_P\Ar(C)\to F$. This induces a map from the realization of the bar resolution to $F$. 

\begin{prop}\label{prop-cofibrant replacement}
The map from the realization of the bar resolution to $F$ is a weak equivalence from a cofibrant object of $\Sp^C$.
\end{prop}

\begin{proof}
In any simplicial model category $\cat{M}$, the realization functor $\on{Fun}(\Delta\op,\cat{M})\to\cat{M}$ is a left Quillen functor by~\cite[VII, Proposition 3.6]{goersssimplicial}. In particular, proposition~\ref{prop-bar is cofibrant} implies  that the realization of the bar resolution is cofibrant. 

The bar resolution with its augmentation to $F$ has an extra degeneracy. This extra degeneracy is given by
\[F\times_P (\Ar(C)\times_P\ldots)\cong F\times_PP\times_P (\Ar(C)\times_P\ldots)\to F\times_P\Ar(C)\times_P (\Ar(C)\times_P\ldots)\]
induced by the unit map $P\to \Ar(C)$. In particular, according to~\cite[Lemma 4.5.1.]{riehlcategorical} the realization of the bar construction is weakly equivalent to $F$. 
\end{proof}

\begin{rem}
Note that the forgetful functors $\Sp^C\to \Sp_{/P}$ and $\Sp_{/P}\to \Sp$ are simplicial and create colimits, hence it does not make a difference to compute the realization in any of these three categories.
\end{rem}

\subsection{Derived Yoneda lemma.}

If $C$ is an internal category whose source map is a fibration, then the category $\Sp^{C\op}$ is a simplicial model category. Let $c\in P$ be an object. We define the internal functor $h_c\in\Sp^{C\op}$. It is given by $h_c=\Ar(C)\times_{P}\{c\}\to P$, where the map $h_c\to P$ is induced by the source map $\Ar(C)\to P$.

Note that if $C$ is actually a simplicial category, then $h_c$ is exactly the presheaf on $C$ represented by $c$. 

We now have a derived Yoneda lemma for internal categories:

\begin{prop}\label{prop-derived Yoneda lemma}
Let $F$ be any object of $\Sp^{C\op}$. Then we have
\[\mathbb{R}\Map_{\Sp^{C\op}}(h_c,F)\simeq \on{hofiber}_cF
\]

\end{prop}

\begin{proof}
The functor $h_c$ is cofibrant in $\Sp^{C\op}$. Both sides of the equations preserve equivalences in the $F$ variable, therefore, it suffices to prove the proposition for $F$ fibrant in $\Sp^{C\op}$, then 
\begin{align*}
\mathbb{R}\Map_{\Sp^{C\op}}(h_c,F)&\simeq \Map_{\Sp^{C\op}}(h_c,F)\\
                                 &\cong\Map_{\Sp_{/P}}(\{c\}\to P,F)\\
                                 &\cong \{c\}\times_PF\\
                                 &\simeq\on{hofiber}_cF
\end{align*}

\end{proof}

\subsection{Weak equivalences between presheaves.}

Let $C$ be a strongly Segal internal category and $f$ be a point of $\Ar(C)$. We have a commutative diagram
\[
\xymatrix{
\Ar(C)\times_P\Ar(C)\ar[r]^{\hspace{10pt}\pi_2}\ar[d]_{m}&\Ar(C)\ar[d]^t&\Delta[0]\ar[l]_{f}\ar[d]\\
\Ar(C)\ar[r]_t&P&\Delta[0]\ar[l]^{t(f)}
}
\]
in which $m$ is the composition map and $\pi_2$ is the second projection. Taking pullbacks on each rows, we get a map
\[(\Ar(C)\times_P\Ar(C))\times_{\Ar(C)}\Delta[0]\to \Ar(C)\times_P\Delta[0]\]

The left hand side is easily seen to coincide with $h_{s(f)}$ while the right hand side is by definition $h_{t(f)}$. Therefore, for any $f\in \Ar(C)$, we have constructed a map $f_*:h_{s(f)}\to h_{t(f)}$.

\begin{lemm}\label{lemm-characterization of equivalences in an internal categories}
Let $f$ be a point in $\Ar(C)$. Then the induced map $f_*:h_{s(f)}\to h_{t(f)}$ is a weak equivalence in $\Sp^{C\op}$ if and only if $f$ is in $C_{hoequiv}$.
\end{lemm}

\begin{proof}
Let $X$ be a Segal space. For $p$ a point in $X_0$, we denote by $h_p$ the fiber product $X_1\times_{X_0}\{p\}$ taken along $d_1:X_1\to X_0$. The map $d_0:X_1\to X_0$ makes $h_p$ into a fibrant space over $X_0$ whose fiber over a point $q$ in $X_0$ is the space $\map_X(q,p)$.

(1) Let $X$ be a Segal space. Recall from~\cite[Section 5.3.]{rezkmodel} that a choice of a section of the Segal map $X_2\to X_1\times_{X_0}X_1$ induces composition maps
\[\map_X(x,y)\times\map_X(y,z)\to \map_X(x,z)\]
that are associative and unital up to homotopy. We assume that such a choice has been made and we denote the corresponding composition by $\circ$. We claim that a point $g$ in $X_1$ is in $X_{hoequiv}$ if and only if the map
\[\map_X(x,d_0g)\goto{g\circ-}\map_X(x,d_1g)\]
given by postcomposition by $g$ is a weak equivalence for each $x$. 

Indeed, if $g$ is a homotopy equivalence, according to~\cite[Section 5.5.]{rezkmodel}, there is a map $h$  in $X_1$ such that $d_0(h)=d_1(g)$ and $d_1(h)=d_0(g)$ and such that $h\circ g$ is in the component of $\id_{d_0(g)}$ in $\map_X(d_0(g),d_0(g))$. Similarly, there is a map $k$ such that $g\circ k$ is in the component of $\id_{d_1(g)}$. Thus, we can consider the composite
\[\map_X(x,d_0g)\goto{g\circ-}\map_X(x,d_1g)\goto{h\circ -} \map_X(x,d_0g)\]

Picking a one simplex in $X_1$ from $h\circ g$ to $\id_{d_0(g)}$, we get a simplicial homotopy from the above composite to the identity map. Therefore $h\circ-$ is a left homotopy inverse for $g\circ-$. Similarly, we would prove that $k\circ -$ is a right homotopy inverse for $g\circ -$.

Conversely, if $g\circ-$ is a weak equivalence, then taking $\pi_0$ and using Yoneda's lemma in $\on{Ho}(X)$, we see that the class of $g$ in $\pi_0(\map_X(d_0g,d_1g))$ is an isomorphism in $\on{Ho}(X)$ which is precisely saying that $g$ is in $X_{hoequiv}$.

(2) Let $g$ be a point in $X_1$. We have a commutative diagram of spaces
\[\xymatrix{
X_1\times_{X_0}X_1\ar[d]_{(d_0\circ\pi_1,\pi_2)}&X_2\ar[l]_{\hspace{20pt}\varphi_2}\ar[d]_{(d_0\circ d_0,d_2)}\ar[r]^{d_1}&X_1\ar[d]^{(d_0,d_1)}\\
X_0\times X_1&X_0\times X_1\ar[r]_{(\id,d_1)}\ar[l]&X_0\times X_0\\
X_0\ar[u]\times\{g\}&X_0\times\{g\}\ar[l]\ar[r]\ar[u]&X_0\times \{d_1g\}\ar[u]
}
\]
in which $\pi_1$ and $\pi_2$ generically denote the left and right projections from a fiber product to its two factors, $\varphi_2$ is the Segal map and the unlabeled maps are either identities or obvious inclusions. Using the fact that $X$ is injectively fibrant, we see that each of the downward pointing arrows is a fibration. Using the fact that $X$ is Segal, we see that each left pointing arrow is a weak equivalence. Therefore, taking pullbacks of each vertical cospan, we get a zig-zag of fibrant objects in $\Sp_{/X_0}$
\[h_{d_0(g)}\stackrel{\sim}{\longleftarrow} Z_g\rightarrow h_{d_1(g)}\]
where $Z_g$ is just a notation for the pullback of the middle cospan.

By proposition~\ref{prop-weak equivalences over a fixed space}, this zig-zag of spaces represents an isomorphism in $\on{Ho}\Sp_{/X_0}$ if and only if for each $q$ in $X_0$, the induced zig-zag on fibers over $q$ is a weak equivalence. In other words, the map $Z_g\to h_{d_1g}$ is a weak equivalence if and only if for each $q$ in $X_0$, the zig-zag
\[\map_X(q,d_0g)\leftarrow {}_qZ_g\to \map_X(q,d_1g)\]
induces an isomorphism in $\on{Ho}(\Sp)$ (where ${}_qZ_g$ denotes the fiber of $Z_g$ over $q$). 

On the other hand, the map ${}_qZ_g\to \map_X(q,d_0g)$ has a preferred section induced by our choice of section of the Segal map $X_2\to X_1\times_{X_0}X_1$. Thus the previous zig-zag represents the same map in $\on{Ho}(\Sp)$ than the map $g\circ-: \map_X(q,d_0g)\to \map_X(q,d_1g)$.

Hence, according to (1), $g$ is in $X_{hoequiv}$ if and only if the map $Z_g\to h_{d_1(g)}$ is a weak equivalence in $\Sp_{/X_0}$.

(3) Now we prove the proposition. Let $i:NC\to X$ be a fibrant replacement in $s\Sp_{inj}$ so that $X$ is a Segal space that is levelwise weakly equivalent to $NC$. We claim that a point $f\in \Ar(C)$ lies in $C_{hoequiv}$ if and only if $i(f)$ is in $X_{hoequiv}$. Indeed, let us consider the following commutative diagram
\[
\xymatrix{
\{f\}\ar[r]\ar[d]&\pi_0(\Ar(C))\ar[d]&\pi_0(C_{hoequiv})\ar[d]\ar[l]\\
\{i(f)\}\ar[r]&\pi_0(X_1)&\pi_0(X_{hoequiv})\ar[l]
}
\]
in which the vertical maps are induced by $i$ and the right pointing horizontal maps send $f$  and  $i(f)$ to their component. 

Since $C_{hoequiv}$ is a set of components of $\Ar(C)$, $f$ is in $C_{hoequiv}$ if and only if the pullback of the top row is non-empty. Similarly, $i(f)$ is in $X_{hoequiv}$ if and only if the pullback of the bottom row is non-empty. Since the vertical maps induce an isomorphism between the top row and the bottom row, their pullbacks must be isomorphic as well.

(4) Let $g=i(f)$. We have a commutative diagram in the category of cospans of spaces
\[
\xymatrix{
[\Ar(C)\times_P\Ar(C)\goto{\pi_2} \Ar(C)\leftarrow \{f\}]\ar[r]^{\hspace{25pt}(m,t,t)}\ar[d]&[\Ar(C)\rightarrow P\leftarrow \{t(f)\}]\ar[d]\\
[X_2\goto{d_2}X_1\leftarrow \{g\}]\ar[r]_{(d_1,d_1,d_1)}&[X_1\rightarrow X_0\leftarrow\{d_1g\}]
}
\]
in which the vertical maps are induced by $i$.

Note that the right pointing map in each cospan is a fibration and the vertical maps are weak equivalences of cospans of spaces.

Thus, taking pullbacks and using the right properness of $\Sp$, we get a commutative diagram of spaces
\[
\xymatrix{
h_{s(f)}\ar[r]^{f_*}\ar[d]&h_{t(f)}\ar[d]\\
Z_g\ar[r]_{\hspace{-15pt}u}& h_{d_1(g)}
}
\]
in which the vertical maps are weak equivalences. Here $Z_g$ is the pullback of the bottom right corner and coincide with $Z_g$ in paragraph (2).

Thus, the map $f_*$ is a weak equivalence if and only if the map $Z_g\to h_{d_1(g)}$ is a weak equivalence. But we have proved in paragraph (2) that this last requirement is equivalent to $g$ being in $X_{hoequiv}$ and according to paragraph (3), this is also equivalent to $f$ being in $C_{hoequiv}$.
\end{proof}

Before stating the following corollary, recall that for a simplicial category $C$, the category $\on{Ho}(C)$ denotes the category obtained by applying $\pi_0$ to each mapping space of $C$.

\begin{coro}\label{coro-equivalence in a simplicial category}
Let $C$ be a fibrant simplicial category. Then, the component of a point $f$ in $\map_C(a,b)$ represents an isomorphism of $\on{Ho}(C)$ if and only if the image of $f$ in $\Ar(C)$ is in $C_{hoequiv}$.
\end{coro}

\begin{proof} 
First, we claim that $f$ represents an isomorphism in $\on{Ho}(C)$ if and only if postcomposition by $f$ induces a weak equivalence $\map_C(x,a)\to\map_C(x,b)$ for any object $x$. 

Indeed, if postcomposition by $f$ induces a weak equivalence as above, then the component of $f$ is an isomorphism in $\on{Ho}(C)$ by Yoneda's lemma in $\on{Ho}(C)$. Conversely, if the component of $f$ is an isomorphism, this means that there exists $g$ a point in $\map_C(b,a)$ and a one-simplex $u$ in $\map_C(a,a)$ connecting $g\circ f$ to $\id_a$ and a one-simplex $v$ in $\map_C(b,b)$ connecting $f\circ g$ to $\id_B$. Then $u$ gives a homotopy between the composite
\[\map_C(x,a)\goto{f\circ-}\map_C(x,b)\goto{g\circ-}\map_C(x,a)\]
and the identity of $\map_C(x,a)$ in the simplicial category of presheaves over $C$. Similarly, $v$ gives a homotopy between the map postcomposing by $f\circ g$ and the identity of $\map_C(x,b)$. This means that postcomposition by $f$ is a weak equivalence for any $x$.

Now, we see that the postcomposition by $f$ induces weak equivalences $\map_C(x,a)\to\map_C(x,b)$ for all $x$ if and only if the map $f_*$ is a fiberwise weak equivalence from $h_a$ to $h_b$. Since the space of objects of $C$ is discrete and $h_a$ and $h_b$ are fibrant in $\Sp^{C}$, $f_*$ is a weak equivalence in $\Sp^{C}$ if and only if it is a fiberwise weak equivalence. Thus, using lemma~\ref{lemm-characterization of equivalences in an internal categories}, we see that the component of $f$ is an isomorphism in $\on{Ho}(C)$ if and only if the image of $f$ in $C$ belongs to $C_{hoequiv}$.  
\end{proof}

\begin{coro}\label{coro-equivalences induce equivalences on presheaves}
Let $C$ be strongly Segal and let $F$ be an object of $\Sp^{C\op}$. Let $f$ be a point in $C_{hoequiv}$. Then, the map 
\[f^*:\on{hofiber}_{t(f)}F\to \on{hofiber}_{s(f)}F\]
is a weak equivalence.
\end{coro}

\begin{proof}
According to the derived Yoneda lemma~\ref{prop-derived Yoneda lemma}, the map $f^*$ can be identified with the map
\[\R\Map_{\Sp^{C\op}}(h_{t(f)},F)\to \R\Map_{\Sp^{C\op}}(h_{s(f)},F)\]
induced by the map $f_*:h_{s(f)}\to h_{t(f)}$. But that map $f_*$ is a weak equivalence according to lemma~\ref{lemm-characterization of equivalences in an internal categories}. Since the derived mapping space preserves weak equivalences, we are done.
\end{proof}

\begin{prop}\label{prop-weak equivalences between presheaves}
Let $C$ be a strongly Segal internal category with $P=\Ob(C)$. Let $S$ be a set of path components of $P$ such that the composite $S\to\pi_0(P)\to \pi_0(P)/\sim$ is surjective. Then a map $u:F\to G$ in $\Sp^{C\op}$ is a weak equivalence if and only if it is a weak equivalence over the path components in $S$.
\end{prop}

\begin{proof}
Clearly, if $u:F\to G$ is a weak equivalence in $\Sp^{C\op}$, the induced map
\[\on{hofiber}_xF\to\on{hofiber}_xG\]
is a weak equivalence for each $x\in P$. In particular, the map $F\to G$ is a weak equivalence over the path components in $S$. 

Conversely, let $u:F\to G$ be a weak equivalence over the path components in $S$. Without loss of generality, we can assume that $F$ and $G$ are fibrant.

Let $c$ be a path component not in $S$. We want to prove that $F\to G$ is a weak equivalence over $c$. Since the map $S\to\pi_0(P)/\sim$ is surjective, there exists a point $f\in C_{hoequiv}$ such that $s(f)$ is in $S$ and $t(f)$ is in $c$.

We have a commutative diagram.
\[
\xymatrix{
F_{t(f)}\ar[d]_{u_{t(f)}}\ar[r]^{f^*}& F_{s(f)}\ar[d]^{u_{s(f)}}\\
G_{t(f)}\ar[r]_{f^*}&G_{s(f)}
}
\]
in which the vertical maps are induced by $u$. The map $u_{s(f)}:F_{s(f)}\to G_{s(f)}$ is a weak equivalence because by assumption $F\to G$ is a weak equivalence over the path components of $s(f)$. On the other hand, the horizontal maps are weak equivalences because of corollary~\ref{coro-equivalences induce equivalences on presheaves}. Therefore, the map $F\to G$ is also a weak equivalence over $c$, the path component of $t(f)$.
\end{proof}

\subsection{Base change adjunction.}

Let $\alpha:C\to D$ be a morphism of internal categories. We want to extract from it an adjunction
\[\alpha_!:\Sp^{C}\leftrightarrows \Sp^{D}:\alpha^*\]

We denote by $P$ the space of objects of $C$ and $Q$ the space of objects of $D$ and by $u:P\to Q$ the value of $\alpha$ on objects.

We start by constructing $\alpha^*$. We have already defined a functor $u^*:\Sp_{/Q\times Q}\to \Sp_{/P\times P}$ and observed that it is lax monoidal. For $F\in \Sp^D$ we have a map
\[(F\times_QP)\times_P u^* \Ar(D)=(F\times_QP)\times_P(P\times_Q\Ar(D)\times_QP)\to F\times_Q\Ar(D)\times_QP\to F\times_QP\]
where the first map is induced by the map $P\to Q$ and the second map is induced by the action of $\Ar(D)$ on $F$.
It is straightforward to check that this equips the object $F\times_QP$ of $\Sp_{/P}$ with an action of the internal category $u^*\Ar(D)$. We can pullback this action along the map $\Ar(C)\to u^*\Ar(D)$ to construct an action of $\Ar(C)$ on $F\times_QP$. The resulting element of $\Sp^C$ is defined to be $\alpha^*F$.

The functor $\alpha_!$ is the left adjoint of $\alpha^*$. It can also be defined as the unique colimit preserving functor sending an internal functor of the form $F\times_P\Ar(C)$ to $F\times_Q\Ar(D)$.

Now assume that $C$ and $D$ are strongly Segal internal categories. If $F\to G$ is a (trivial) fibration in $\Sp^D$, then $P\times_QF\to P\times_QG$ is a (trivial) fibration in $\Sp^C$. Thus $\alpha^*$ is a right Quillen functor.

This allows us to define the homotopy colimit and more generally the homotopy left Kan extension of an internal functor.

\begin{defi}
Let $\alpha:C\to D$ be a map between strongly Segal internal categories, and let $F\in \Sp^C$. The homotopy left Kan extension of $F$ along $\alpha$ is the left derived functor of $\alpha_!$ applied to $F$.
\end{defi}

By proposition~\ref{prop-cofibrant replacement}, $\mathbb{L}\alpha_!F$ can be computed as the realization of the following simplicial object in $\Sp^D$
\[F\times_Q\Ar(D)\leftleftarrows F\times_P\Ar(C)\times_Q\Ar(D) \mathrel{\substack{\textstyle\leftarrow\\[-0.6ex]
                      \textstyle\leftarrow \\[-0.6ex]
                      \textstyle\leftarrow}} F\times_P\Ar(C)\times_P\Ar(C)\times_Q\Ar(D)\ldots\]
  
\begin{rem}                    
If $C$ and $D$ are fibrant simplicial categories, the simplicial object we constructed to compute the homotopy left Kan extension coincides with the usual bar construction.
\end{rem}

We have the following derived version of a classical fact in category theory.

\begin{prop}\label{prop-left Kan extension is fully faithful}
Let $C$ and $D$ be strongly Segal internal categories. Let $\alpha:C\to D$ be a fully faithful map. Then, for any object $F$ of $\Sp^C$, the derived unit map $F\to \R\alpha^*\L\alpha_! F$ is a weak equivalence. 
\end{prop}

\begin{proof}
We denote by $Q$ the space of objects of $D$, and by $P$ the space of objects of $C$ and by $u:P\to Q$ the map induced by $\alpha$ on objects. 

(1) Let $P\to P'\to Q$ be a factorization of $P\to Q$ as a weak equivalence followed by a fibration. Since $\alpha$ is fully faithful and $P'\times P'\to Q\times Q$ is a fibration, the map
\[\Ar(C)\to P'\times_Q\Ar(D)\times_Q P'\]
is a weak equivalence. This map factors as
\[\Ar(C)\to P\times_Q\Ar(D)\times_QP'\to P'\times_Q\Ar(D)\times_QP'\]
Since the source map $\Ar(D)\to Q$ is a fibration, the second map is a weak equivalence which implies by the two-out-of-three property that the map
\[\Ar(C)\to P\times_Q\Ar(D)\times_QP'\]
is a weak equivalence.

(2) The map $\Ar(C)\to P\times_Q\Ar(D)\times_QP'$ constructed in (1) is a weak equivalence of spaces over $P$ where on the left the map to $P$ is the source map and on the right it is the first projection. Moreover, we claim that this map is a map in $\Sp^{C\op}$. The left action of $\Ar(C)$ on $P\times_Q\Ar(D)\times_QP'$ is obtained by noticing that $P\times_Q\Ar(D)\times_QP'=\alpha^*(\Ar(D)\times_QP')$.

(3) If $F$ is an object of $\Sp^C$ and $G$ is a fibrant object of $\Sp^{C\op}$, we denote by $B_\bullet(F,C,G)$ the simplicial space
\[B_\bullet(F,C,G)=F\times_PG\leftleftarrows F\times_{P}\Ar(C)\times_PG \mathrel{\substack{\textstyle\leftarrow\\[-0.6ex]
                      \textstyle\leftarrow \\[-0.6ex]
                      \textstyle\leftarrow}} F\times_{P}\Ar(C)\times_{P}\Ar(C)\times_PG\ldots\]
                      
A weak equivalence $G\to G'$ between fibrant objects of $\Sp^{C\op}$ induces a levelwise equivalence $B_\bullet(F,C,G)\to B_\bullet(F,C,G')$ and hence a weak equivalence between their geometric realizations since by~\cite[Theorem 15.8.7.]{hirschhornmodel} any simplicial space is Reedy cofibrant. In particular, using (2), we find a weak equivalence
\[|B_\bullet(F,C,\Ar(C))|\to |B_\bullet(F,C,P\times_Q\Ar(D)\times_QP')|\]

This can be composed with the weak equivalence $F\to |B_\bullet(F,C,\Ar(C))|$ constructed in proposition~\ref{prop-cofibrant replacement}. In the end we get a weak equivalence
\[F\to |B_\bullet(F,C,P\times_Q\Ar(D)\times_QP')|\]

(4) Now we prove that for any $F$ in $\Sp^C$ the derived unit map $F\to \R\alpha^*\L\alpha_!F$ is a weak equivalence. Since weak equivalences in $\Sp^C$ are weak equivalences of the underlying spaces, it suffices to check that the map $F\to \R u^*\L\alpha_!F$ is a weak equivalence of spaces where $u^*:\Sp_{/Q}\to \Sp_{/P}$ is the functor sending $X\to Q$ to $X\times_QP$.

Recall that $p:P'\to Q$ is the fibration factoring $u:P\to Q$ as a weak equivalence followed by a fibration. We denote by $p'$ the functor sending $K\to Q$ to $K\times_QP'$. For $K$ a space over $Q$, we have an obvious pullback square
\[
\xymatrix{
u^*K\ar[d]\ar[r]& p^*K\ar[d]\\
P\ar[r]& P'
}
\]
The functor $p^*$ preserves all weak equivalences since $\Sp$ is right proper and $P'\to Q$ is a fibration. Moreover, if $X\to Q$, is a fibration, then $p^*X\to P'$ is a fibration and the previous pullback square is a homotopy pullback square. Since its bottom map is a weak equivalence, this implies that the natural transformation $u^*\to p^*$ is a weak equivalence on fibrant objects of $\Sp_{/Q}$. 

Let $R$ be a fibrant replacement functor in $\Sp_{/Q}$. We have a commutative diagram of functors $\Sp_{/Q}\to\Sp$
\[
\xymatrix{
u^*\ar[d]\ar[r]& u^*\circ R\ar[d]^{\simeq}\\
p^*\ar[r]^{\hspace{-10pt}\simeq}& p^*\circ R
}
\]
Thus the map $F\to u^*(R\L \alpha_! F)$ is a weak equivalence if and only if the map $F\to p^*\L\alpha_!F$ is a weak equivalence.

But we know that a model for $\L\alpha_! F$ is the realization of $B_\bullet(F,C,P\times_Q\Ar(D))$. Moreover, in spaces geometric realization commute with base change. Therefore, the map $F\to p^*\L\alpha_!F$ can be identified up to weak equivalence with the map
\[F\to |B_\bullet(F,C,p^*P\times_Q\Ar(D))|=|B_\bullet(F,C,P\times_Q\Ar(D)\times_QP')|\]
which according to (3) is a weak equivalence.
\end{proof}

\begin{theo}\label{theo-invariance of presheaf category}
Let $C$ and $D$ be strongly Segal internal categories. Let $\alpha:C\to D$ be a Rezk equivalence. Then the Quillen adjunction
\[\alpha_!:\Sp^{C}\leftrightarrows \Sp^D:\alpha^*\]
is a Quillen equivalence.
\end{theo}

\begin{proof}
As in the previous proposition, we denote by $Q$ the space of objects of $D$, and by $P$ the space of objects of $C$ and by $u:P\to Q$ the map induced by $\alpha$ on objects.

First, since $C$ and $D$ are strongly Segal, they are Segal fibrant which implies by proposition~\ref{prop-Dwyer Kan equivalences are Rezk equivalences ICat} that $\alpha$ is a Dwyer-Kan equivalence. Thus, we know from proposition~\ref{prop-left Kan extension is fully faithful} that the derived unit is an equivalence. Let $F\to Q$ be an object of $\Sp^D$. We want to prove that $\mathbb{L}\alpha_!\R\alpha^*F\to F$ is an equivalence. Since $\alpha$ is essentially surjective, according to proposition~\ref{prop-weak equivalences between presheaves}, it suffices to check that this map is a weak equivalence over the set $S$ of components of $Q$ containing a point of the form $u(p)$ for some point $p$ of $P$. 

Let $p$ be a point of $P$ and $h^p$ be the object of $\Sp^{C}$ corepresented by $u(p)$, that is $h^p=\{p\}\times_P\Ar(C)$. By definition of $\alpha_!$, $\alpha_! h^p= \{p\}\times_Q\Ar(D)=h^{u(p)}$. Since $h^p$ is cofibrant, $\alpha_!h^p$ is weakly equivalent to $\L\alpha_!h^p$.

Thus, in order to prove the proposition, it suffices to check that for any $p$ in $P$, the map
\[\R\Map_{\Sp^D}(\L\alpha_!h^p,\mathbb{L}\alpha_!\R\alpha^*F)\to \R\Map_{\Sp^D}(\L\alpha_!h^p,F)\]
is a weak equivalence. By proposition~\ref{prop-derived adjunction}, it is equivalent to prove that the map
\[\R\Map_{\Sp^C}(h^p,\R\alpha^*\mathbb{L}\alpha_!\R \alpha^*F)\to \R\Map_{\Sp^C}(h^p,\R\alpha^*F)\]
is a weak equivalence. But this follows immediately from proposition~\ref{prop-left Kan extension is fully faithful}.
\end{proof}

\begin{example}
In this example we allow ourselves to treat topological spaces as simplicial sets. The reader is invited to apply the functor $\on{Sing}$ as needed.

Let $(K,k)$ be a connected based topological space. There is a strongly Segal internal category $\on{Path}(K)$ described in~\cite[Example II.3.4]{andrademanifolds} whose objects are points of $K$ and morphisms are Moore paths between points. There is an obvious map $K\to\on{Path}(K)$ where $K$ is the subcategory of constant paths. This map is a levelwise equivalence. On the other hand, there is a Dwyer-Kan equivalence $\Omega K\to \on{Path}(K)$ where $\Omega K$ is the space of endomorphisms of $k$ in $\on{Path}(K)$. It is a strictly associative model for the loop space of $K$ at $k$. Hence we have a zig-zag of left Quillen equivalences
\[\Sp_{/K}\to\Sp^{\on{Path}(K)}\leftarrow \Sp^{\Omega K}\]
This recovers the folk theorem that spaces over $K$ are equivalent to spaces with an action the Moore loops of $K$.
\end{example}

\section{Comparison with simplicial categories}

In this last section, we compare the homotopy theory of internal categories with respect to the Rezk equivalences with that of simplicially enriched categories with respect to the Dwyer-Kan equivalences (defined in~\cite{bergnermodel}). Unfortunately, there does not seem to be a Quillen adjunction relating these two model categories. Nevertheless, we prove that the two underlying relative categories are equivalent. We recall a few facts about relative categories in the first subsection.

\subsection{Relative categories.}

\begin{defi}
A relative category is a pair $(\cat{C},w\cat{C})$ where $\cat{C}$ is a category and $w\cat{C}$ is a subcategory containing all the objects.
\end{defi}

The arrows of the category $w\cat{C}$ are called the weak equivalences of $\cat{C}$.

Note that any model category is in particular a relative category if we drop the data of the cofibrations and fibrations. Relative categories can be used to encode the homotopy theory of infinity categories by the following theorem.

\begin{theo}[Barwick-Kan]
There is a model category structure on the category of small relative categories in which the weak equivalences are the maps that are sent to weak equivalences in $\cat{CSS}_{inj}$ by Rezk's relative nerve construction (defined at the beginning of section 8 of~\cite{rezkmodel}). Moreover this model structure is Quillen equivalent to $\cat{CSS}_{inj}$.
\end{theo}

\begin{proof}
This is the main theorem of~\cite{barwickrelative}.
\end{proof}

The following result will be our main tool to prove that certain maps are weak equivalences of relative categories

\begin{prop}\label{prop-equivalence relative}
Let $(F,wF):(\cat{C},w\cat{C})\to(\cat{D},w\cat{D})$ and $(G,wG):(\cat{D},w\cat{D})\to(\cat{C},w\cat{C})$ be two maps between relative categories. Assume that there exists a zig-zag of natural transformations between $FG$ and $\id_\cat{D}$ and another zig-zag of natural transformations between $GF$ and $\id_{\cat{C}}$. Assume further that both zig-zags are objectwise zig-zags of weak equivalences. Then $F$ and $G$ are weak equivalences of relative categories. 
\end{prop}

\begin{proof}
This follows from~\cite[Proposition 7.5.]{barwickrelative} together with the fact that homotopy equivalences between simplicial spaces are levelwise (and in particular Rezk) weak equivalences. 
\end{proof}

If $F$ and $G$ satisfy the condition of the previous proposition, we will say that $G$ is a homotopy inverse of $F$. Note that not all weak equivalences of relative categories admit a homotopy inverse. 

\begin{coro}\label{coro-Quillen equivalences are relative equivalences}
Let $F:\cat{X}\leftrightarrows \cat{Y}:G$ be a Quillen equivalence between cofibrantly generated model categories. Let $Q$ be a cofibrant replacement functor on $\cat{X}$ and $R$ be a fibrant replacement functor on $\cat{Y}$, then $FQ$ and $GR$ are equivalences of relative categories.
\end{coro}

\begin{proof}
By definition of a Quillen equivalence, the map $FQGR(Y)\to FGR(Y)\to RY$ induced by the natural transformation $Q\to \id_{\cat{X}}$ and the counit $FG\to \id_{\cat{D}}$ is a weak equivalence which is natural in $Y$, thus we have a functorial zig-zag of weak equivalences
\[FQGR(Y)\to RY\leftarrow Y\]

Similarly, we have a functorial zig-zag of weak equivalences
\[X\leftarrow QX\to GRFQ(X)\]

Thus the conditions of proposition~\ref{prop-equivalence relative} are satisfied.
\end{proof}

\subsection{The internalization functor.}

The category $\Cat_{\Delta}$ of simplicial categories is the full subcategory of $\ICat$ spanned by internal categories whose space of objects is discrete. In this section, the inclusion functor $\Cat_{\Delta}\to \ICat$ shall be denoted $\Int$ (for internalization).

If $C$ is a simplicial category, we denote by $\map_C(x,y)$ the fiber of the structure map $C\to\Ob(C)\times\Ob(C)$ over the point $(x,y)$. We denote by $\on{Ho}(C)$ the ordinary category obtained by applying $\pi_0$ to each mapping space.

\begin{defi}
A map $f:C\to D$ between simplicial categories is said to be
\begin{itemize}
\item essentially surjective if the induced map $\on{Ho}(C)\to\on{Ho}(D)$ is essentially surjective.
\item fully faithful if for each pair $(x,y)$ of objects of $C$, the induced map
\[\map_C(x,y)\to\map_D(f(x),f(y))\]
is a weak equivalence.
\item a Dwyer-Kan equivalence if it is both essentially surjective and fully faithful.
\end{itemize}
\end{defi}

\begin{rem}
It is proved in~\cite{bergnermodel} that the category $\Cat_{\Delta}$ has a model structure in which the weak equivalences are the Dwyer-Kan equivalences and the fibrant objects are the simplicial categories whose mapping spaces are fibrant simplicial sets.
\end{rem}

\begin{prop}
Let $f:C\to D$ be a map between fibrant simplicial categories. Then $f$ is
\begin{enumerate}
\item fully faithful if and only if $\Int(f)$ is fully faithful.
\item essentially surjective if and only if $\Int(f)$ is essentially surjective.
\end{enumerate}
\end{prop}

\begin{proof}
Observe first that if $C$ is a fibrant simplicial category, then $\Int(C)$ is a Segal fibrant internal category. Without this observation, the proposition would not make sense.

(1) By definition, the map $\Int(f)$ is fully faithful if and only if  the square
\[
\xymatrix{
\Ar(\Int(C))\ar[r]\ar[d]_{(s,t)}& \Ar(\Int(D))\ar[d]^{(s,t)}\\
\Ob(C)\times \Ob(C)\ar[r]&\Ob(D)\times \Ob(D)
}
\]
is homotopy cartesian. Observe that both vertical maps in this square are fibrations. Thus, by proposition~\ref{prop-characterization homotopy cartesian}, $\Int(f)$ is fully faithful if and only if for any $(c,d)\in \Ob(C)\times\Ob(C)$, the induced map $\map_C(c,d)\to\map_D(fc,fd)$ is a weak equivalence, which is exactly saying that the map $f$ is fully faithful.

(2) The essential surjectivity of $f$ is equivalent to the surjectivity of the induced map on isomorphism classes $\on{Ho}(C)/\cong\to\on{Ho}(D)/\cong$ while the essential surjectivity of $\Int(f)$ is equivalent to the surjectivity of 
\[\pi_0(\Ob(\Int(C)))/\sim\to\pi_0(\Ob(\Int(D)))/\sim.\]

Thus in order to show that the two notions coincide, it suffices to show that the two functors $C\mapsto \on{Ho}(C)/\cong $ and $C\mapsto \pi_0(\Ob(\Int(C)))/\sim$ are naturally isomorphic. There is a surjective map $\Ob(C)\to \pi_0(\Ob(\Int(C)))/\sim$. According to corollary~\ref{coro-equivalence in a simplicial category}, this functor induces an isomorphism
\[\on{Ho}(C)/\cong\goto{}\pi_0(\Ob(\Int(C)))/\sim\]
This isomorphism is obviously natural in $C$.
\end{proof}

\begin{prop}\label{prop-Int reflects weak equivalences}
The functor $\Int:\Cat_{\Delta}\to \ICat$ preserves and reflects weak equivalences.
\end{prop}

\begin{proof}
For $C$ a simplicial category, there is a fibrant replacement $C'$ with the same set of objects (one can apply a product preserving fibrant replacement in $\Sp$ to each mapping space like $\on{Ex}^{\infty}$). The map $C\to C'$ is a levelwise equivalence on nerve, hence $\Int(C)\to\Int(C')$ is a Rezk equivalence.

Let $f:C\to D$ be a map in $\Cat_{\Delta}$. We can include it in a diagram
\[
\xymatrix{
C\ar[r]^f\ar[d]&D\ar[d]\\
C'\ar[r]^{f'}&D'
}
\]
in which each of the vertical maps is a fibrant replacement as above. The functor $\Int$ applied to the vertical maps yields levelwise weak equivalences. Thus, $\Int(f)$ is a Rezk equivalence if and only if $\Int(f')$ is a Rezk equivalence. According to the previous proposition, the map $\Int(f')$ is a Dwyer-Kan equivalence if and only if the map $f'$ is a Dwyer-Kan equivalence. Since $\Int(C')$ and $\Int(D')$ are Segal fibrant internal categories, we see that $f'$ is a Dwyer-Kan equivalence if and only if $\Int(f')$ is a Rezk equivalence. Since Dwyer-Kan equivalences in $\Cat_{\Delta}$ satisfy the two-out-of-three property, $f'$ is a Dwyer-Kan equivalence if and only if $f$ is one. If we put together these three equivalences, we have proved that $f$ is a Dwyer-Kan equivalence if and only if $\Int(f)$ is a Rezk equivalence. 
\end{proof}

\subsection{The equivalence of relative categories.}

We have proved in proposition~\ref{prop-Int reflects weak equivalences}, that $\Int$ sends Dwyer-Kan equivalences to Rezk equivalences thus we can see $\Int$ as a map of relative categories.
\[\Int:\Cat_{\Delta}\to\ICat\]

\begin{theo}\label{theo-Int is an equivalence}
The functor $\Int$ induces a weak equivalence of relative categories.
\end{theo}

\begin{proof}
In order to prove this result, we will use Bergner's work in~\cite{bergnerthree}. We recall the definition of a Segal precategory. This is a diagram $\Delta\op\to\Sp$ whose value at $[0]$ is a discrete space (i.e. a constant simplicial set). 

In this proof, given a model category $\cat{X}$, we denote by $\cat{X}^f$ the relative categories of fibrant objects of $\cat{X}$ with the induced weak equivalences.

(1) If $F:\cat{X}\leftrightarrows \cat{Y}:G$ is a Quillen equivalence between cofibrantly generated model categories, the composite $\cat{Y}^f\to\cat{Y}\goto{G}\cat{X}$ is an equivalence of relative categories. Indeed, it has a homotopy inverse given by $RFQ:\cat{X}\to\cat{Y}^f$. We can apply this observation to the identity adjunction $\cat{Y}\leftrightarrows\cat{Y}$ and we find that the inclusion $\cat{Y}^f\to\cat{Y}$ is a weak equivalence of relative categories.

(2) Weak equivalences of relative categories satisfy the two-out-of-three properties, thus according to (1), it suffices to prove that the inclusion $\Int:\Cat^f_{\Delta}\to \ICat$ is a weak equivalence. Using again the two-out-of-three property and corollary~\ref{coro-Quillen equivalences are relative equivalences}, it suffices to check that the composite
\[N\circ\Int:\Cat^f_{\Delta}\to\cat{CSS}_{proj}\]
is an equivalence of relative categories.

(3) Bergner in~\cite[Theorem 8.6.]{bergnerthree} shows that the obvious inclusion $\Cat_{\Delta}\to \cat{SeCat}_{f}$ from the category of simplicial categories to the category of Segal precategories with the projective model structure (defined in~\cite[Theorem 7.1.]{bergnerthree}) is a right Quillen equivalence. This implies by (1) that the induced map $\Cat_{\Delta}^f\to \cat{SeCat}_{f}$ is a weak equivalence of relative categories.

(4) Similarly Bergner shows in~\cite[Theorem 6.3.]{bergnerthree} that the obvious inclusion $\cat{SeCat}_{c}\to\cat{CSS}_{inj}$ is a left Quillen equivalence. The model category $\cat{SeCat}_c$ is an other model structure on Segal precategories constructed in~\cite[Theorem 5.1.]{bergnerthree} in which all objects are cofibrant but with the same weak equivalences as $\cat{SeCat}_f$. Thus, by corollary~\ref{coro-Quillen equivalences are relative equivalences}, the inclusion $\cat{SeCat}_{c}\to\cat{CSS}_{inj}$ is a weak equivalence of relative categories. Since $\cat{SeCat}_c=\cat{SeCat}_f$ and $\cat{CSS}_{inj}=\cat{CSS}_{proj}$ as relative categories, we see that the inclusion $\cat{SeCat}_f\to\cat{CSS}_{proj}$ is a weak equivalence of relative categories.

(5) Coming back to (2), the map $N\circ\on{Int}:\Cat_\Delta^f\to\cat{CSS}_{proj}$ coincides with the composite of the two inclusions
\[\Cat_{\Delta}^f\to\cat{SeCat}_{f}\to\cat{CSS}_{proj}\]
and we have seen in (3) and (4) that both maps are weak equivalences of relative categories.
\end{proof}

\bibliographystyle{alpha}
\bibliography{biblio}

\end{document}